\documentclass[11pt]{article}
\usepackage[a4paper,margin=1in]{geometry}
\usepackage{amsmath,amsthm,aliascnt}
\usepackage[hidelinks]{hyperref}
\usepackage[nameinlink,capitalize,noabbrev]{cleveref}
\usepackage{amssymb,booktabs,microtype}
  \newtheorem{theorem}{Theorem}[section]
  \newaliascnt{proposition}{theorem}
  \newtheorem{proposition}[proposition]{Proposition}
  \aliascntresetthe{proposition}
  \newaliascnt{lemma}{theorem}
  \newtheorem{lemma}[lemma]{Lemma}
  \aliascntresetthe{lemma}
  \newaliascnt{corollary}{theorem}
  \newtheorem{corollary}[corollary]{Corollary}
  \aliascntresetthe{corollary}
  \newaliascnt{conjecture}{theorem}
  
  \aliascntresetthe{conjecture}
  \newaliascnt{problem}{theorem}
  \newtheorem{problem}[problem]{Problem}
  \aliascntresetthe{problem}
  \theoremstyle{remark}
  \newaliascnt{remark}{theorem}
  \newtheorem{remark}[remark]{Remark}
  \aliascntresetthe{remark}
  \theoremstyle{plain}
\hypersetup{hidelinks,
 pdftitle={A Sharp Spectral Mantel Theorem for Quantum Transport},
 pdfauthor={Xingkun Song},
 pdfkeywords={triangle-free graph, spectral extremal problem, graphon, flag algebra, quantum transport}}
\crefname{theorem}{Theorem}{Theorems}
\Crefname{theorem}{Theorem}{Theorems}
\crefname{proposition}{Proposition}{Propositions}
\Crefname{proposition}{Proposition}{Propositions}
\crefname{lemma}{Lemma}{Lemmas}
\Crefname{lemma}{Lemma}{Lemmas}
\crefname{corollary}{Corollary}{Corollaries}
\Crefname{corollary}{Corollary}{Corollaries}
\crefname{conjecture}{Conjecture}{Conjectures}
\Crefname{conjecture}{Conjecture}{Conjectures}
\crefname{problem}{Problem}{Problems}
\Crefname{problem}{Problem}{Problems}
\crefname{remark}{Remark}{Remarks}
\Crefname{remark}{Remark}{Remarks}

\newcommand{\Tr}{\operatorname{Tr}}
\newcommand{\homd}{t}
\newcommand{\e}{\mathrm e}
\newcommand{\ii}{\mathrm i}
\newcommand{\cT}{\mathcal T}
\newcommand{\Id}{\operatorname{Id}}
\newcommand{\opnorm}[1]{\lVert #1\rVert_{\mathrm{op}}}
\newcommand{\HSnorm}[1]{\lVert #1\rVert_{\mathrm{HS}}}

\title{A Sharp Spectral Mantel Theorem for Quantum Transport}
\author{Xingkun Song\thanks{School of Mathematics and Statistics, Qinghai Minzu University, and Qinghai Institute of Applied Mathematics, Xining, Qinghai 810007, P.R. China (\email{xksong@126.com}).}}
\providecommand{\email}[1]{\href{mailto:#1}{\texttt{#1}}}
\date{}
\begin{document}
\maketitle

\begin{abstract}
We study the maximum average quantum transition probability between
distinct vertices of a triangle-free graph, equivalently the minimum
average return probability, at times inversely proportional to the order.
For every compact interval of positive scaled times below
$\tau_{\mathrm c}\approx3.79099$, the balanced complete bipartite graph
is the unique maximizer for all sufficiently large orders.
The threshold is sharp: it solves
$\tau_{\mathrm c}=4\sin(\tau_{\mathrm c}/2)$, and the balanced complete
bipartite graph is eventually suboptimal at every larger scaled time.
The corresponding graphon functional combines the edge density with
alternating even cycle densities.  We prove that the balanced complete
bipartite graphon is its unique maximizer through the critical time.
On a nonempty interval immediately afterward, the unique maximizer is a
balanced bipartite graphon with constant cross-edge weight below one.
After normalization by the squared time, we obtain uniform square-root
$L^1$ stability through the critical endpoint and at zero time.
We also solve the bipartite problem at every positive time and establish
sharp cut distance stability.  The proofs combine spectral interpolation,
an exactly certified six-vertex moment inequality, vertex-measure
variation, and a minimum-degree argument.  A uniform finite approximation
bound and vertex cloning connect the graphon and exact finite results.
\end{abstract}



\section{Introduction}\label{sec:introduction}

Mantel's theorem states that a triangle-free graph on $n$ vertices has at
most $\lfloor n^2/4\rfloor$ edges, with equality only for the balanced
complete bipartite graph \cite{Mantel1907}. Spectral extremal graph theory
asks how such forbidden-subgraph conditions constrain eigenvalues; see
\cite{Nikiforov2011} for a survey. Nosal's inequality bounds the spectral
radius of a triangle-free graph by the square root of its number of edges
\cite{Nosal1970}, and Nikiforov extended this inequality to clique-free
graphs \cite{Nikiforov2002}. Further spectral extremal results appear in
\cite{CioabaDesaiTait2024,TaitTobin2017,WangKangXue2023}. These results
suggest studying which extremal structures persist when the objective
depends on more of the spectrum. We consider a time-dependent objective
arising from quantum walks whose dense limit involves every even
spectral moment of order at least four, together with the edge density.

Continuous-time quantum walks were introduced in the study of quantum
algorithms by Farhi and Gutmann \cite{FarhiGutmann1998} and provide models
of coherent transport on networks \cite{MuelkenBlumen2011}. For such a
walk, average return probability measures the probability of finding an
excitation at its initial vertex, averaged over the choice of that vertex.
M\"ulken and Blumen \cite{MuelkenBlumen2006} used spectral lower bounds on
this quantity to study transport efficiency; return probabilities also
play a central role in the analysis of honeycomb networks
\cite{BatalhaEtAl2022}. These two works use Laplacian Hamiltonians.
For adjacency-driven walks, which we study here, minimizing average
return is equivalently maximizing the average transition probability
between distinct vertices. This objective concerns transport away from
a sampled source at a specified time, while perfect state transfer
concerns a prescribed pair of vertices \cite{Godsil2012}.

We ask which triangle-free graphs optimize this average transition
probability. The connection with Mantel's theorem appears on the dense
time scale $t=\tau/n$: after normalization, the limiting objective tends
to the edge density as $\tau\to0$. At positive scaled times, all even
cycle densities enter, with alternating coefficients. Thus the question
is whether the edge-maximizing graph remains optimal when these higher
moments contribute, and where its optimality ends. Short-time expansions
for quantum walks are studied in \cite{Szigeti2019}; here the extremal
problem requires control of the full series at fixed scaled time.

We determine the exact threshold for the persistence of the balanced
complete bipartite graphon and identify its fractional bipartite
replacement on an interval immediately beyond that threshold. Below the
threshold, the corresponding finite extremal graph is exactly the
balanced complete bipartite graph for all sufficiently large orders.
The results below also give uniform stability through the critical time,
a classification within the bipartite class at every positive time, and
quantitative bounds connecting finite graphs to the limiting problem.

\subsection{Exact finite extremality}

For $n\ge2$, let $\cT_n$ be the family of triangle-free graphs on $n$
vertices and set
\[
 F_G(t)=\frac{1}{n(n-1)}
 \sum_{u\ne v}|(\e^{-\ii tA(G)})_{vu}|^2.
\]
Define the average return and non-return probabilities by
\[
 R_G(t)=\frac1n\sum_u|(\e^{-\ii tA(G)})_{uu}|^2,
 \qquad E_G(t)=1-R_G(t)=(n-1)F_G(t),
\]
where the last identity follows from unitarity. To interpret $E_G(t)$,
choose and record a uniformly random source $u$, prepare $|u\rangle$,
and measure the vertex after evolution for time $t$. Then $E_G(t)$ is
the probability that the outcome differs from $u$; averaging detection
at a uniformly chosen distinct target gives $F_G(t)$. The measurement
may follow earlier departures and returns, so this is an instantaneous
non-return probability.

We use the adjacency Hamiltonian, with unit coupling on every edge and
no restriction on the total number of edges beyond triangle-freeness.
The scaling $t=\tau/n$ is equivalently evolution under $A(G)/n$ for time
$\tau$. Since $\opnorm{A(G)/n}\le1$, it gives a common spectral scale for
dense graphs of different orders. Set
\begin{equation}\label{eq:Phi-n}
 \Phi_n(G,\tau)=n^2F_G(\tau/n).
\end{equation}
Let $\tau_{\mathrm c}$ be the positive solution in $(0,2\pi)$ of
\begin{equation}\label{eq:critical-equation}
 \tau=4\sin(\tau/2),
 \qquad \tau_{\mathrm c}\approx3.79099.
\end{equation}

\begin{theorem}\label{thm:finite-exactness}
Let $I\subset(0,\tau_{\mathrm c})$ be nonempty and compact.  There is
$n_0(I)$ such that, for every $n\ge n_0(I)$ and $\tau\in I$, the balanced
complete bipartite graph
$T_2(n)=K_{\lfloor n/2\rfloor,\lceil n/2\rceil}$ is the unique maximizer,
up to isomorphism, of $F_G(\tau/n)$ over $G\in\cT_n$.
Conversely, for each fixed $\tau>\tau_{\mathrm c}$, the graph $T_2(n)$ is
not a maximizer whenever $n$ is sufficiently large.
\end{theorem}

The order threshold $n_0(I)$ is not made explicit, and exact finite
extremality at $\tau=\tau_{\mathrm c}$ remains open.  The proof first
locates a finite maximizer near the complete bipartite graphon, then uses
vertex cloning and edge completion to recover the exact graph.

\subsection{The graphon transition}

The dense graph limit formulation makes the higher moments explicit.
For a finite graph $H$ and a graphon $W$, write $\homd(H,W)$ for the
homomorphism density of $H$ in $W$. Graphons describe limits of dense
graph sequences through these densities
\cite{Lovasz2012,LovaszSzegedy2006}. The limiting functional is
\begin{equation}\label{eq:Phi-series}
 \Phi_\tau(W)=\tau^2\homd(K_2,W)
 +2\sum_{r=2}^{\infty}
 \frac{(-1)^{r+1}\tau^{2r}}{(2r)!}\homd(C_{2r},W).
\end{equation}
Let
\begin{equation}\label{eq:Mtriangle}
 M_\triangle(\tau)
 =\sup\{\Phi_\tau(W): W\text{ is a graphon and }\homd(K_3,W)=0\},
\end{equation}
and define $B_q$, for $q\in[0,1]$, by splitting $[0,1]$ into two sets of
measure $1/2$, assigning value $q$ across the two sets, and assigning value
zero within them.  Thus $B_1$ is the balanced complete bipartite graphon.
We use weak isomorphism in the standard sense of cut distance zero; precise
graphon definitions are given in \cref{sec:dense-bridge}.
\begin{theorem}\label{thm:global}
For $0\le\tau\le\tau_{\mathrm c}$,
\begin{equation}\label{eq:global-sharp}
 M_\triangle(\tau)=4\left(1-\cos\frac\tau2\right).
\end{equation}
If $0<\tau\le\tau_{\mathrm c}$ and a triangle-free graphon $W$ satisfies
$\Phi_\tau(W)=M_\triangle(\tau)$, then $W$ is weakly isomorphic to $B_1$.
For every $\tau>\tau_{\mathrm c}$, the graphon $B_1$ is not a maximizer.
\end{theorem}

For $\tau>\tau_{\mathrm c}$, let $q_\tau\in(0,1)$ be the unique solution of
\begin{equation}\label{eq:q-equation}
 \frac{\tau^2}{2}(1-2q_\tau)
 +2\tau\sin\left(\frac{\tau q_\tau}{2}\right)=0.
\end{equation}
Existence and uniqueness, as well as the classification of bipartite
maximizers for every positive time, are proved in \cref{thm:threshold}.
The same fractional graphon is the global triangle-free maximizer just
beyond the threshold.

\begin{theorem}\label{thm:second-phase}
There is $\varepsilon_0>0$ such that, for every
$\tau_{\mathrm c}<\tau<\tau_{\mathrm c}+\varepsilon_0$,
\begin{equation}\label{eq:second-phase}
 M_\triangle(\tau)
 =\frac{\tau^2}{2}q_\tau(1-q_\tau)
   +4\left(1-\cos\frac{\tau q_\tau}{2}\right),
\end{equation}
and $B_{q_\tau}$ is the unique triangle-free maximizer up to weak
isomorphism.
\end{theorem}

The graphon $B_{q_\tau}$ is realized as a cut limit of balanced bipartite
simple graphs that are quasirandom across their parts, with cross-edge
density tending to $q_\tau$
\cite{ChungGrahamWilson1989,LovaszSos2008}. Thus, immediately beyond the
threshold, such graphs have a larger leading transport coefficient than
the complete bipartite graph. The parameter $q_\tau$ describes edge
density, with every present edge retaining unit coupling. The largest
interval on which this structure is globally optimal is not determined here.

\subsection{Stability and proof strategy}

Write
\[
 d_1(W,B_1)=\inf_\varphi\lVert W^\varphi-B_1\rVert_1,
\]
where the infimum is over measure-preserving bijections of $[0,1]$.
To include the small-time limit, define
\begin{equation}\label{eq:normalized-functional}
 \Psi_\tau(W)=\frac{\Phi_\tau(W)}{\tau^2}\quad(\tau>0),
 \qquad \Psi_0(W)=\homd(K_2,W).
\end{equation}
The series \eqref{eq:Phi-series} makes this extension continuous at zero.
\begin{theorem}\label{thm:stability}
There is a constant $C>0$ such that, for every triangle-free graphon $W$
and every $0\le\tau\le\tau_{\mathrm c}$,
\begin{equation}\label{eq:normalized-L1-stability}
 d_1(W,B_1)\le C
 \sqrt{\Psi_\tau(B_1)-\Psi_\tau(W)}.
\end{equation}
The exponent $1/2$ is best possible even for cut distance, at every fixed
$\tau\in[0,\tau_{\mathrm c}]$.
\end{theorem}

Critical stability is also an input to \cref{thm:second-phase}.
A vertex-measure variation shows that extremizers near the critical time
have minimum degree greater than $2/5$.  The
Andr\'asfai--Erd\H{o}s--S\'os theorem \cite{AndrasfaiErdosSos1974}
then forces bipartiteness.

Within the bipartite class, put $q_*(\tau)=1$ for
$0<\tau\le\tau_{\mathrm c}$ and $q_*(\tau)=q_\tau$ otherwise.
The optimal cut distance stability exponent is $1/2$ for
$0<\tau\le\tau_{\mathrm c}$ and $1/4$ for $\tau>\tau_{\mathrm c}$;
see \cref{thm:bip-stability}.
The finite extremal values approximate $M_\triangle$ with an explicit
$O(n^{-1})$ error on bounded time intervals
(\cref{thm:finite-dense}).  Together with \cref{thm:stability}, this gives
quantitative bounds for finite near-maximizers through the critical time
(\cref{cor:finite-stability}).

The global bound follows by combining a sixth-degree spectral minorant
with inequalities for the edge, $C_4$, and $C_6$ densities.  The
four-vertex inequality also gives $L^1$ stability.  The six-vertex
inequality in \cref{thm:flag-c246} is certified by rational matrices and
an exact-arithmetic verifier supplied with the paper.  For bipartite
graphons, a singular value argument reduces the extremal problem to
$B_q$, and a fourth-moment estimate controls the distance to this family.
The certificate uses the flag algebra framework
\cite{Razborov2007}; related arguments for pentagons in triangle-free
graphs appear in \cite{Grzesik2012,HatamiEtAl2013}.

\Cref{sec:finite-graphs} establishes the finite-to-graphon bridge.
\Cref{sec:dense-limits,sec:extremal-inequalities} develop the spectral and
moment estimates, and \cref{sec:triangle-free-problem} proves the global
bound, stability, and finite approximation.  The bipartite problem is
settled in \cref{sec:bipartite-problem}.  The global result beyond the
threshold and exact finite extremality are proved in
\cref{sec:second-phase,sec:finite-exactness}, respectively.
\Cref{sec:further-problems} records the remaining questions.

\section{From finite transport to the graphon functional}
\label{sec:finite-graphs}

\subsection{Basic identities}

For a finite graph $H$, write $e(H)=|E(H)|$.  Throughout,
$G=(V,E)$ is a finite simple undirected graph with $|V|=n$, adjacency matrix
$A$, degrees $d_u$, and $m=e(G)$.  Put
\[
   U(t)=U_G(t)=\e^{-\ii tA}.
\]
We write $\Tr$ for the trace, $\Id$ for the identity matrix or operator,
and $\opnorm{\cdot}$ for the operator norm on the ambient space.  For a
matrix $M$, let $M^*$ denote its conjugate
transpose; its Hilbert--Schmidt norm is
\[
 \HSnorm{M}^2=\Tr(M^*M)=\sum_{u,v}|M_{uv}|^2.
\]
Since $U_G(t)$ is unitary, the matrix
$U_G(t)\circ\overline{U_G(t)}$ is doubly stochastic, where $\circ$ denotes
entrywise multiplication.  Thus the average transition probability over
all ordered vertex pairs is always $1/n$; compare the average mixing matrix
studied in \cite{Godsil2013}.  We therefore omit the diagonal pairs and set
\[
 S_G(t)=\sum_{u\ne v}|U_G(t)_{vu}|^2,
 \qquad
 F_G(t)=\frac{S_G(t)}{n(n-1)}.
\]

\begin{lemma}\label{lem:basic-identities}
For every finite simple graph $G$,
\begin{align}
 S_G(t)
 &=n-\sum_{u\in V}|U(t)_{uu}|^2, \label{eq:escape}\\
 S_G(t)
 &=2\Tr\bigl(\Id-\cos(tA)\bigr)
   -\sum_{u\in V}|(U(t)-\Id)_{uu}|^2. \label{eq:trace-defect}
\end{align}
In particular, $S_G(t)$ is nonnegative, even, and real analytic on
$\mathbb R$.
\end{lemma}

\begin{proof}
The squared Euclidean norm of every column of $U(t)$ is one.  Summing the
squared moduli of the off-diagonal entries by columns gives \eqref{eq:escape}.  For
$K(t)=U(t)-\Id$, the off-diagonal entries of $K(t)$ and $U(t)$ agree, whence
\[
 S_G(t)=\HSnorm{K(t)}^2-\sum_u|K(t)_{uu}|^2.
\]
Since $U(t)$ is unitary,
\[
 \HSnorm{U(t)-\Id}^2
 =\Tr\bigl((U(t)-\Id)^*(U(t)-\Id)\bigr)
 =2n-2\operatorname{Re}\Tr U(t),
\]
which is $2\Tr(\Id-\cos(tA))$.  This proves \eqref{eq:trace-defect}.
The right side of \eqref{eq:escape} is even and real analytic, and it is
nonnegative by definition.
\end{proof}

Equation \eqref{eq:trace-defect} separates a spectral trace term from a
diagonal defect.  The defect vanishes uniformly under the dense scaling
below.

\subsection{The dense bridge}\label{sec:dense-bridge}

A graphon is a symmetric measurable function from $[0,1]^2$ to $[0,1]$.
Such a function may take fractional values.  Lebesgue measure is denoted by $\mu$, and
$\mathbf1_S$ denotes the indicator of a measurable set $S$.  If $G$ has
vertex set $[n]$, partition $[0,1]$
into intervals $I_1,\ldots,I_n$ of measure $1/n$ and define its step graphon
by
\[
 W_G(x,y)=\mathbf 1_{\{ij\in E(G)\}}
 \quad\text{for }x\in I_i,\ y\in I_j.
\]
For a finite graph $H$, set
\[
 \homd(H,W)=\int_{[0,1]^{V(H)}}
   \prod_{ij\in E(H)}W(x_i,x_j)\prod_{i\in V(H)}dx_i,
 \qquad
 \homd(H,G)=\frac{\hom(H,G)}{n^{|V(H)|}}.
\]
Here $\hom(H,G)$ is the number of graph homomorphisms from $H$ to $G$.
Then $\homd(H,G)=\homd(H,W_G)$.

For an integrable kernel $K$, write
\[
 \lVert K\rVert_\square
 =\sup_{S,T\subseteq[0,1]\text{ measurable}}
   \left|\int_{S\times T}K(x,y)\,dx\,dy\right|.
\]
The cut distance $\delta_\square(U,W)$ is the infimum of
$\lVert U-W^\varphi\rVert_\square$ over measure-preserving bijections
$\varphi$ of $[0,1]$, understood modulo null sets, where
$W^\varphi(x,y)=W(\varphi(x),\varphi(y))$.  Graphons $U$ and $W$ are
\emph{weakly isomorphic} if $\delta_\square(U,W)=0$; equivalently,
$\homd(H,U)=\homd(H,W)$ for every finite simple graph $H$
\cite{Lovasz2012,LovaszSzegedy2006}.

We write $G_n\to W$ when
$\delta_\square(W_{G_n},W)\to0$; equivalently,
$\homd(H,G_n)\to\homd(H,W)$ for every finite graph $H$
\cite{Lovasz2012,LovaszSzegedy2006}.

A graphon is triangle-free if $\homd(K_3,W)=0$.  It is bipartite if there is
a measurable partition $[0,1]=X\sqcup Y$ such that $W=0$ almost everywhere
on $X^2\cup Y^2$.

The graphon $W$ also defines an integral operator on $L^2[0,1]$ by
\[
 (T_Wf)(x)=\int_0^1W(x,y)f(y)\,dy.
\]
Since $W$ belongs to $L^2([0,1]^2)$, the operator $T_W$ is
Hilbert--Schmidt, with
\begin{equation}\label{eq:HS-kernel}
 \HSnorm{T_W}^2
 =\int_{[0,1]^2}W(x,y)^2\,dx\,dy<\infty;
\end{equation}
symmetry of $W$ makes $T_W$ self-adjoint, and every Hilbert--Schmidt
operator is compact; see \cite[Sec.~7.5]{Lovasz2012}.
We write $\mathbf1=\mathbf1_{[0,1]}$ for the constant-one function.
We abbreviate $\int_{[0,1]^2}W(x,y)^2\,dx\,dy$ by $\int W^2$.

We now use the dense scaling $t=\tau/n$.  By \eqref{eq:Phi-n},
\[
 \Phi_n(G,\tau)=\frac{n}{n-1}S_G(\tau/n).
\]

\begin{lemma}\label{lem:defect}
For every simple graph $G$ on $n\ge2$ vertices and every real $\tau$,
\begin{align}
 S_G(\tau/n)
 &=2\Tr\left(\Id-\cos\left(\frac{\tau A}{n}\right)\right)-D_G(\tau),
 \label{eq:dense-trace}\\
 0\le D_G(\tau)
 &\le\frac{\tau^4}{4n^4}\sum_u d_u^2\le\frac{\tau^4}{4n}.
 \label{eq:defect-bound}
\end{align}
If $G$ is triangle-free, then $D_G(\tau)\le\tau^4/(16n)$.
\end{lemma}

\begin{proof}
By \eqref{eq:trace-defect},
\[
 D_G(\tau)=\sum_u
 \left|\left(\e^{-\ii\tau A/n}-\Id\right)_{uu}\right|^2.
\]
For $f_u(t)=(\e^{-\ii tA})_{uu}$, we have $f_u(0)=1$,
$f_u'(0)=0$, and
\[
 |f_u''(t)|
 =|\langle Ae_u,\e^{-\ii tA}Ae_u\rangle|
 \le\lVert Ae_u\rVert_2^2=d_u,
\]
where $e_u$ is the coordinate vector at $u$.  Taylor's integral remainder
therefore gives $|f_u(t)-1|\le t^2d_u/2$.  Squaring and summing with
$t=\tau/n$ proves the first bound in \eqref{eq:defect-bound}; the second
uses $\sum_u d_u^2\le n(n-1)^2\le n^3$.
If $G$ is triangle-free, then $d_u+d_v\le n$ for every edge $uv$, whence
\[
 \sum_u d_u^2=\sum_{uv\in E(G)}(d_u+d_v)
 \le ne(G)\le n^3/4
\]
by Mantel's theorem.  This gives the last assertion.
\end{proof}

For $n\ge2$, on the subspace of functions that are constant on each $I_i$, the operator
$T_{W_G}$ is represented by $A(G)/n$ and vanishes on the orthogonal
complement.  Since $W_G$ is $\{0,1\}$-valued, the power series in
\eqref{eq:Phi-series} gives
\[
 \Phi_\tau(W_G)
 =2\Tr\left(\Id-\cos\left(\frac{\tau A(G)}n\right)\right).
\]
Consequently, \cref{lem:defect} yields the exact relation
\begin{equation}\label{eq:finite-graphon-bridge}
 \Phi_n(G,\tau)
 =\frac{n}{n-1}\bigl(\Phi_\tau(W_G)-D_G(\tau)\bigr).
\end{equation}
Thus $\Phi_n$ and $\Phi_\tau$ are different at finite order.  Since
$0\le\Phi_\tau(W_G)\le\tau^2$, the exact bridge and
\eqref{eq:defect-bound} give
\begin{equation}\label{eq:bridge-error}
 \left|\Phi_n(G,\tau)-\Phi_\tau(W_G)\right|
 \le
 \frac{\tau^2+\tau^4/4}{n-1},
 \qquad n\ge2.
\end{equation}
In particular, the difference tends to zero locally uniformly in $\tau$,
uniformly over all $n$-vertex graphs $G$. The separate $K_2$-term in
\eqref{eq:Phi-series} reflects simple graph approximation: a closed walk
of length two traverses one edge twice and contributes the edge density
$\int W$, whereas the second operator moment of a fractional graphon is
$\int W^2$. This explains why replacing the edges of a complete bipartite
graph by couplings of strength $q$ does not give the objective associated
with a simple graph sequence converging to $B_q$.

As $\tau\to0$, the first terms of \eqref{eq:Phi-series} are
\begin{equation}\label{eq:first-moments}
 \Phi_\tau(W)
 =\tau^2\homd(K_2,W)
 -\frac{\tau^4}{12}\homd(C_4,W)
 +\frac{\tau^6}{360}\homd(C_6,W)
 +O(\tau^8),
\end{equation}
where the remainder is uniform over all graphons.

\begin{theorem}\label{thm:graphon-limit}
Let $|V(G_n)|=n\to\infty$ and suppose that
$\delta_\square(W_{G_n},W)\to0$ for a graphon $W$.  Then, locally uniformly
for $\tau\in\mathbb R$,
\begin{equation}\label{eq:graphon-functional}
 \Phi_n(G_n,\tau)\longrightarrow\Phi_\tau(W).
\end{equation}
If every $G_n$ is triangle-free, then $W$ is triangle-free.
\end{theorem}

\begin{proof}
Write $A=A(G_n)$.
Expanding the trace in \eqref{eq:dense-trace} gives
\[
 2\Tr\left(\Id-\cos\left(\frac{\tau A}{n}\right)\right)
 =2\sum_{r\ge1}\frac{(-1)^{r+1}\tau^{2r}}{(2r)!}
   \frac{\Tr A^{2r}}{n^{2r}}.
\]
For $r=1$,
\[
 \frac{\Tr A^2}{n^2}=\frac{2e(G_n)}{n^2}
 =\homd(K_2,G_n)\longrightarrow\homd(K_2,W).
\]
For $r\ge2$, $\Tr A^{2r}=\hom(C_{2r},G_n)$, so
\[
 \frac{\Tr A^{2r}}{n^{2r}}=\homd(C_{2r},G_n)
 \longrightarrow\homd(C_{2r},W).
\]
The normalized number of closed walks is at most one.  Fix $T_0>0$.  On
$|\tau|\le T_0$, the
$r$th summand is therefore bounded in absolute value by
$2T_0^{2r}/(2r)!$, a summable sequence independent of $n$.  Together with
\cref{lem:defect} and $n/(n-1)\to1$, uniform convergence of the dominated series proves
\eqref{eq:graphon-functional} locally uniformly in $\tau$.  If every $G_n$ is
triangle-free, then
$\homd(K_3,G_n)=0\to\homd(K_3,W)$, so $W$ is triangle-free.
\end{proof}

For a convergent sequence as in \cref{thm:graphon-limit}, the definitions
give, for each fixed $\tau$,
\[
 F_{G_n}(\tau/n)=\frac{\Phi_\tau(W)+o(1)}{n^2},
 \qquad
 E_{G_n}(\tau/n)=\frac{\Phi_\tau(W)+o(1)}{n}.
\]
The non-return probability therefore vanishes on this time scale, and
$\Phi_\tau(W)$ gives its leading coefficient. Maximizing $\Phi_\tau$
selects graph structures that asymptotically minimize average return at
the prescribed scaled time. The first terms in \eqref{eq:first-moments}
show how the edge-density objective of Mantel's theorem is modified by
cycle densities as the time increases. The spectral representation below
allows these contributions to be treated together.

\section{Spectral representation and triangle-free stability}
\label{sec:dense-limits}

\begin{proposition}\label{prop:spectral-form}
For every graphon $W$,
\begin{equation}\label{eq:spectral-form}
 \Phi_\tau(W)
 =\tau^2\left(\homd(K_2,W)-\int W^2\right)
 +2\Tr\bigl(\Id-\cos(\tau T_W)\bigr).
\end{equation}
Equivalently, with both sums taken over the nonzero eigenvalues $\lambda_j$
of $T_W$, counted with multiplicity,
\[
 \Phi_\tau(W)
 =\tau^2\left(\homd(K_2,W)-\sum_j\lambda_j^2\right)
 +2\sum_j\bigl(1-\cos(\tau\lambda_j)\bigr).
\]
If $W$ is $\{0,1\}$-valued almost everywhere, the first term in
\eqref{eq:spectral-form} vanishes.
\end{proposition}

\begin{proof}
By \eqref{eq:HS-kernel},
\[
 \Tr T_W^2=\HSnorm{T_W}^2=\int W^2.
\]
The operator $T_W^{2r}$ is trace class for every $r\ge1$, and
\[
 \Tr T_W^{2r}=\homd(C_{2r},W),\qquad r\ge2.
\]
The estimate
$0\le1-\cos(\tau x)\le\tau^2x^2/2$ and
$\sum_j\lambda_j^2<\infty$ show that
$\sum_j(1-\cos(\tau\lambda_j))$ converges absolutely.  Thus
$\Id-\cos(\tau T_W)$ is trace class; the trace in
\eqref{eq:spectral-form} is taken of this difference, not of its two terms
separately.  Moreover,
\[
 \sum_j\sum_{r\ge1}\frac{|\tau\lambda_j|^{2r}}{(2r)!}
 \le(\cosh|\tau|-1)\sum_j\lambda_j^2<\infty,
\]
because $|\lambda_j|\le\opnorm{T_W}\le1$.  Hence the trace series may be
summed termwise, and
\[
 2\Tr(\Id-\cos(\tau T_W))
 =\tau^2\int W^2
 +2\sum_{r\ge2}\frac{(-1)^{r+1}\tau^{2r}}{(2r)!}
   \homd(C_{2r},W).
\]
Comparison with \eqref{eq:Phi-series} proves \eqref{eq:spectral-form}.
\end{proof}

The finite precursor of the next lemma is Nosal's spectral form of Mantel's
theorem: a triangle-free graph with $m$ edges has spectral radius at most
$\sqrt m$ \cite{Nosal1970}.  Nikiforov subsequently obtained a sharp
clique-free extension and the corresponding equality characterization
\cite{Nikiforov2002}.  We prove the graphon bound directly; its quantitative equality statement
is given in \cref{lem:quantitative-mantel}.

\begin{lemma}\label{lem:spectral-mantel}
Let $W$ be triangle-free, put $p=\homd(K_2,W)$, and let $\lambda_j$ run
through the nonzero eigenvalues of $T_W$, counted with multiplicity.  Then
\begin{equation}\label{eq:spectral-Mantel}
 p\le\lambda_{\max}(T_W)\le\sqrt{\frac p2}\le\frac12,
 \qquad
 |\lambda_j|\le\lambda_{\max}(T_W)\quad\text{for every }j.
\end{equation}
\end{lemma}

\begin{proof}
If $W=0$ almost everywhere, then \eqref{eq:spectral-Mantel} is trivial.
Otherwise, let $\rho=\lambda_{\max}(T_W)$.  The constant function has norm
one, so $\rho\ge\langle T_W\mathbf1,\mathbf1\rangle=p>0$.  For every
$f\in L^2[0,1]$,
\[
 |\langle T_Wf,f\rangle|
 \le \langle T_W|f|,|f|\rangle
 \le \rho\lVert f\rVert_2^2.
\]
The variational characterization of the extreme eigenvalues therefore gives
$|\lambda_j|\le\rho$ for every $j$.  Triangle-freeness implies
\[
 \sum_j\lambda_j^3=\Tr T_W^3=\homd(K_3,W)=0.
\]
Consequently,
\[
 \sum_{\lambda_j<0}|\lambda_j|^3
 =\sum_{\lambda_j>0}\lambda_j^3\ge\rho^3.
\]
Because $|\lambda_j|\le\rho$,
\[
 \sum_{\lambda_j<0}\lambda_j^2
 \ge\frac1\rho\sum_{\lambda_j<0}|\lambda_j|^3\ge\rho^2.
\]
The eigenvalue $\rho$ contributes another $\rho^2$, and therefore
\[
 2\rho^2\le\sum_j\lambda_j^2=\int W^2\le\int W=p.
\]
Since $\rho\ge p$, we obtain $2p^2\le2\rho^2\le p$ and hence $p\le1/2$.

\end{proof}

The density bound has the following quantitative form, in the spirit
of classical Tur\'an stability \cite{Furedi2015}.  Its explicit constants
will be used in \cref{thm:stability}.

\begin{lemma}\label{lem:quantitative-mantel}
Let $W$ be a triangle-free graphon and write
\[
 p=\homd(K_2,W)=\frac12-\eta,
 \qquad 0\le\eta\le\frac12.
\]
There is a measurable partition $[0,1]=X\sqcup Y$ such that $W=0$ almost
everywhere on $X^2$ and the following properties hold.  Let $K_{X,Y}$ be one on $X\times Y\cup Y\times X$ and zero
elsewhere, and write $\lVert K\rVert_1=\int_{[0,1]^2}|K|$.  Then
\begin{align}
 \int_{X^2\cup Y^2}W&\le2p\eta,\label{eq:mantel-internal}\\
 \left|\mu(X)-\frac12\right|^2&\le\eta(1-\eta),
 \label{eq:mantel-balance}\\
 \lVert W-K_{X,Y}\rVert_1&\le3\eta.
 \label{eq:mantel-L1}
\end{align}
Consequently,
\begin{equation}\label{eq:mantel-cut-stability}
 \inf_\varphi\lVert W^\varphi-B_1\rVert_1\le2\sqrt\eta,
 \qquad
 \delta_\square(W,B_1)\le2\sqrt\eta,
\end{equation}
where the infimum is over measure-preserving bijections of $[0,1]$.
In particular, $p=1/2$ if and only if $W$ is weakly isomorphic to $B_1$.
\end{lemma}

No optimality is claimed for the numerical constant $2$ in
\eqref{eq:mantel-cut-stability}.  The exponent $1/2$, however, is sharp, as
shown in \cref{rem:sqrt-sharpness}.  The proof is deferred to
\cref{app:technical-estimates}.

\subsection{A uniform truncation estimate}

\begin{lemma}\label{lem:cycle-bound}
If $W$ is triangle-free, then for every $r\ge2$,
\[
 0\le\homd(C_{2r},W)\le 2^{1-2r}.
\]
\end{lemma}

\begin{proof}
Let $\lambda_j$ run through the nonzero eigenvalues of $T_W$, counted with
multiplicity.  By \cref{lem:spectral-mantel}, $|\lambda_j|\le1/2$ and
\[
 \sum_j\lambda_j^2=\int W^2\le\int W=\homd(K_2,W)\le\frac12.
\]
Hence
\[
 \homd(C_{2r},W)=\sum_j\lambda_j^{2r}
 \le\left(\frac12\right)^{2r-2}\sum_j\lambda_j^2
 \le2^{1-2r}.
\]
\end{proof}

For an integer $\ell\ge2$, define the truncated transport
\begin{equation}\label{eq:truncation}
 \Phi_\tau^{[\ell]}(W)
 =\tau^2\homd(K_2,W)
 +2\sum_{r=2}^{\ell}\frac{(-1)^{r+1}\tau^{2r}}{(2r)!}
 \homd(C_{2r},W).
\end{equation}

\begin{proposition}\label{prop:tail}
For every triangle-free graphon $W$, every integer $\ell\ge2$, and every
real $\tau$,
\begin{equation}\label{eq:tail-bound}
 \left|\Phi_\tau(W)-\Phi_\tau^{[\ell]}(W)\right|
 \le4\sum_{r=\ell+1}^{\infty}\frac{(|\tau|/2)^{2r}}{(2r)!}.
\end{equation}
\end{proposition}

\begin{proof}
By \cref{lem:cycle-bound},
\begin{align*}
 \left|\Phi_\tau(W)-\Phi_\tau^{[\ell]}(W)\right|
 &\le2\sum_{r=\ell+1}^{\infty}
   \frac{|\tau|^{2r}}{(2r)!}\homd(C_{2r},W)\\
 &\le4\sum_{r=\ell+1}^{\infty}
   \frac{(|\tau|/2)^{2r}}{(2r)!}.
\end{align*}
\end{proof}

\begin{corollary}\label{cor:continuity-attainment}
Suppose that $\tau_n\to\tau$ and that triangle-free graphons $W_n$ converge
to a triangle-free graphon $W$ in cut distance.  Then
\[
 \Phi_{\tau_n}(W_n)\longrightarrow\Phi_\tau(W).
\]
Consequently, $M_\triangle(\tau)$ is attained for every $\tau\in\mathbb R$,
and $M_\triangle$ is continuous on $\mathbb R$.
\end{corollary}

\begin{proof}
Every finite homomorphism density is continuous in the cut metric
\cite[Chap.~10]{Lovasz2012} and \cite{LovaszSzegedy2006}.  The finite truncations in
\eqref{eq:truncation} are therefore jointly continuous in $\tau$ and $W$.
On each compact $\tau$-interval, \cref{prop:tail} makes the convergence of
these truncations uniform over all triangle-free graphons.  This proves the
first assertion.

Because $\homd(K_3,\cdot)$ is cut-continuous, the triangle-free graphon space
modulo weak isomorphism is a closed subspace of the compact graphon space.
The first assertion therefore gives a
maximizer for each $\tau$.  If $\tau_n\to\tau$, compactness applied to
maximizers at $\tau_n$ gives the upper semicontinuity of $M_\triangle$;
evaluating at a maximizer for $\tau$ gives the reverse inequality.  Hence
$M_\triangle$ is continuous.
\end{proof}

\section{Triangle-free moment inequalities and spectral bounds}
\label{sec:extremal-inequalities}

\subsection{The candidate graphons}

Recall the graphons $B_q$ defined in the introduction.  For $0<q<1$,
$B_q$ is not the
step graphon of a finite simple graph, but it is a cut limit of finite
balanced bipartite graphs.  The step graphons of the balanced complete
bipartite graphs converge to $B_1$ in cut distance.

\begin{proposition}\label{prop:bq}
For $0\le q\le1$,
\begin{equation}\label{eq:bq}
 \Phi_\tau(B_q)
 =\frac{\tau^2}{2}q(1-q)
 +4\left(1-\cos\frac{\tau q}{2}\right).
\end{equation}
In particular,
\[
 \Phi_\tau(B_1)=4\left(1-\cos\frac\tau2\right).
\]
\end{proposition}

\begin{proof}
For $q=0$ the assertion is immediate. For $q>0$, we have
$\homd(K_2,B_q)=q/2$ and $\int B_q^2=q^2/2$, and the only nonzero
eigenvalues of $T_{B_q}$ are $q/2$ and $-q/2$. Substitution in
\eqref{eq:spectral-form} gives \eqref{eq:bq}.
\end{proof}

\begin{lemma}\label{lem:critical-sign}
The function $\psi(\tau)=4\sin(\tau/2)-\tau$ has exactly one positive zero,
namely $\tau_{\mathrm c}$.  It is positive on
$(0,\tau_{\mathrm c})$ and negative on $(\tau_{\mathrm c},\infty)$.
\end{lemma}

\begin{proof}
Since $\psi'(\tau)=2\cos(\tau/2)-1$, the function increases on
$(0,2\pi/3)$ and decreases on $(2\pi/3,2\pi)$.  It is positive for small
positive $\tau$, while $\psi(2\pi)=-2\pi$, so it has exactly one positive
zero in $(0,2\pi)$.  Finally, $\psi(\tau)\le4-\tau<0$ for $\tau\ge2\pi$.
\end{proof}

\subsection{The spectral minorant and coefficient comparison}

We now return to the scalar function in the spectral representation.  The
polynomial below is the unique cubic in $y=x^2$ that agrees with
$L(\tau\sqrt y)$ to first order at both endpoints of $[0,1/4]$.  This choice
preserves equality at the spectrum of $B_1$ while retaining only the fourth
and sixth moments.

Put
\begin{equation}\label{eq:hermite-coefficients}
 L(z)=z^2-2(1-\cos z),\qquad
 \begin{aligned}
 a_\tau&=48L(\tau/2)-4\tau L'(\tau/2),\\
 b_\tau&=16\tau L'(\tau/2)-128L(\tau/2).
 \end{aligned}
\end{equation}

\begin{lemma}\label{lem:hermite}
If \(0<\tau\le4\) and \(|x|\le1/2\), then
\begin{equation}\label{eq:hermite-bound}
 L(\tau x)\ge a_\tau x^4+b_\tau x^6.
\end{equation}
Equality holds if and only if \(x=0\) or \(|x|=1/2\).
\end{lemma}

\begin{proof}
Let \(\Lambda_\tau(y)=L(\tau\sqrt y)\) for \(0\le y\le1/4\), and let
\(P_\tau(y)=a_\tau y^2+b_\tau y^3\).  The definitions in
\eqref{eq:hermite-coefficients} give
\[
 P_\tau(0)=\Lambda_\tau(0),\quad
 P_\tau'(0)=\Lambda_\tau'(0),\quad
 P_\tau(1/4)=\Lambda_\tau(1/4),\quad
 P_\tau'(1/4)=\Lambda_\tau'(1/4).
\]
Moreover,
\[
 \Lambda_\tau^{(4)}(y)
 =2\sum_{r=4}^{\infty}(-1)^r
 \frac{r(r-1)(r-2)(r-3)\tau^{2r}y^{r-4}}{(2r)!}.
\]
At \(y=0\), only the positive term with \(r=4\) remains.  If
\(0<y\le1/4\), then \(\tau^2y\le4\), and the ratio of successive absolute
terms is at most
\[
 4\frac{r+1}{r-3}\frac1{(2r+2)(2r+1)}
 \le\frac29,\qquad r\ge4.
\]
The alternating series therefore gives \(\Lambda_\tau^{(4)}(y)>0\).
The Hermite remainder formula now yields, for \(0<y<1/4\),
\[
 \Lambda_\tau(y)-P_\tau(y)
 =\frac{\Lambda_\tau^{(4)}(\xi)}{4!}\,y^2(y-1/4)^2>0
\]
for some \(\xi\in(0,1/4)\).  This proves the assertion after setting
\(y=x^2\).
\end{proof}

\begin{lemma}\label{lem:coefficient-signs}
For \(0<\tau\le\tau_{\mathrm c}\), set
\begin{equation}\label{eq:gamma-definition}
 \gamma_\tau=\tau\left(4\sin\frac\tau2-\tau\right),
\end{equation}
and
\begin{equation}\label{eq:coefficient-definitions}
 \alpha_\tau=-\frac{75}{38}b_\tau,\qquad
 \beta_\tau=a_\tau+\frac{357}{152}b_\tau.
\end{equation}
Then
\[
 \tau^2-a_\tau-\frac38b_\tau=\gamma_\tau,
\]
and \(\alpha_\tau>0\) and \(\beta_\tau>0\).  Moreover,
\(\gamma_\tau>0\) for \(0<\tau<\tau_{\mathrm c}\), while
\(\gamma_{\tau_{\mathrm c}}=0\).
\end{lemma}

The proof is an alternating series estimate and is deferred to
\cref{app:technical-estimates}.

\subsection{Triangle-free moment inequalities}

We prove a four-vertex inequality and then the six-vertex inequality
stated in \cref{thm:flag-c246}.  Together they give the coefficient
conditions used below.

For a finite graph $H$ of order $k$, let $t_{\rm ind}(H,W)$ denote its
\emph{unlabeled induced density}: the probability that the random graph
obtained from $k$ independent uniform points, with edges inserted
independently with probabilities $W(x_i,x_j)$, is isomorphic to $H$.  Thus
the values over all unlabeled graphs of order $k$ sum to one.  This is
$k!/|\operatorname{Aut}(H)|$ times the induced density for one fixed
labeling of $H$.

\begin{lemma}\label{lem:four-cycle-stability}
Let $W$ be triangle-free.  Write
\[
 p=\homd(K_2,W),\quad \eta=\frac12-p,\quad
 \Delta_4(W)=\homd(C_4,W)-p+\frac38,
\]
and put $d(x)=\int_0^1W(x,y)\,dy$ and
$v=\int_0^1(d(x)-p)^2\,dx$.  Then
\begin{align}
 \Delta_4(W)&\ge\frac34\eta^2+\frac14v,
 \label{eq:four-cycle-degree-bound}\\
 d_1(W,B_1)&\le\sqrt{140\Delta_4(W)}.
 \label{eq:four-cycle-L1-bound}
\end{align}
In particular, $\homd(C_4,W)\ge p-3/8$.
\end{lemma}

\begin{proof}
Define the nonnegative quantities
\begin{align*}
 A&=\int_{[0,1]^2}(1-W(x,y))(1-d(x)-d(y))^2\,dx\,dy,\\
 B&=\int_{[0,1]^2}(1-W(x,y))(d(x)-d(y))^2\,dx\,dy,\\
 C&=\int_{[0,1]^2}W(x,y)(d(x)-d(y))^2\,dx\,dy.
\end{align*}
The four-vertex identity is
\begin{equation}\label{eq:flag-certificate}
 \Delta_4(W)=\frac38A+\frac58B+\frac18C
 +\frac7{48}t_{\rm ind}(K_2\sqcup2K_1,W)
 +\frac{11}{24}t_{\rm ind}(2K_2,W).
\end{equation}
It can be checked directly by expanding the integrands.  In detail, if
$u_j=\int d(x)^j\,dx$ and $z=\int W(x,y)d(x)d(y)\,dx\,dy$, then
\[
 A=1-5p+6u_2+2p^2-2u_3-2z,\qquad
 B=2u_2-2p^2-2u_3+2z,\qquad C=2u_3-2z.
\]
Expanding the nonedge factors in the two induced densities, with every
term containing a triangle equal to zero, gives
\begin{align*}
 t_{\rm ind}(K_2\sqcup2K_1,W)
 &=6p-24u_2+12u_3+36z-6p^2-12\homd(C_4,W),\\
 t_{\rm ind}(2K_2,W)&=3p^2-12z+6\homd(C_4,W).
\end{align*}
These formulas prove \eqref{eq:flag-certificate} with the stated unlabeled
normalization.

We have $B+C=2v$.  Moreover,
\begin{align*}
 \int(1-W(x,y))(1-d(x)-d(y))\,dx\,dy
 &=1-3p+2\int d(x)^2\,dx\\
 &\ge(1-p)(1-2p)=2(1-p)\eta.
\end{align*}
Weighted Cauchy--Schwarz gives $A\ge4(1-p)\eta^2$.
Since $p\le1/2$, \eqref{eq:flag-certificate} implies
\[
 \Delta_4(W)\ge\frac38A+\frac18(B+C)
 \ge\frac34\eta^2+\frac14v.
\]

Choose the partition in \cref{lem:quantitative-mantel}, and put
$a=\mu(X)$, $b=\mu(Y)$, $r=\int_{Y^2}W$, and $s=|a-1/2|$.
Thus $W=0$ on $X^2$ almost everywhere, $r\le2p\eta\le\eta$, and
$\lVert W-K_{X,Y}\rVert_1\le3\eta$.
If $a\ge1/2$, then $d(x)\le b$ on $X$, so
\[
 a(s-\eta)=a(p-b)
 \le\int_X(p-d(x))\,dx\le\sqrt{av}.
\]
If $a\le1/2$, then $\int_Yd\le ab+r$, so
\[
 b(s-\eta)=b(p-a)
 \le r+\int_Y(p-d(x))\,dx\le r+\sqrt{bv}.
\]
Dividing by the part measure, which is at least $1/2$ in each case,
gives $s\le3\eta+\sqrt{2v}$.  Balancing the partition changes
$K_{X,Y}$ in $L^1$ by $2s(1-s)$.  Hence
\begin{align*}
 d_1(W,B_1)&\le3\eta+2s(1-s)\le9\eta+2\sqrt{2v}\\
 &\le\sqrt{140}\sqrt{\frac34\eta^2+\frac14v}
 \le\sqrt{140\Delta_4(W)}.
\end{align*}
The penultimate step is Cauchy--Schwarz, since
$9^2/(3/4)+(2\sqrt2)^2/(1/4)=140$.
\end{proof}

No optimality is claimed for the constant $140$.

The six-vertex inequality is stated next.  Its coefficients
are chosen so that, together with \cref{lem:four-cycle-stability} and the density
bound $p\le1/2$, it controls the spectral minorant throughout the sharp
interval.  The coefficient comparison is given in
\cref{prop:spectral-functional-criterion}.

\begin{theorem}\label{thm:flag-c246}
If \(W\) is triangle-free and \(p=\homd(K_2,W)\), then
\begin{equation}\label{eq:flag-c246}
 881-2400p+2856\homd(C_4,W)-1216\homd(C_6,W)\ge0.
\end{equation}
\end{theorem}

\begin{proof}
Let \(\mathcal H_6\) be the set of the \(38\) isomorphism classes of
triangle-free graphs on six vertices, and put
\(p_6(H;W)=t_{\rm ind}(H,W)\).  For \(k\in\{4,6\}\), let
\[
 N_k(H)=\frac{\operatorname{inj}(C_k,H)}{2k},
\]
where \(\operatorname{inj}(C_k,H)\) is the number of injective
homomorphisms from \(C_k\) to \(H\).  Thus \(N_k(H)\) counts unoriented,
not necessarily induced, copies of \(C_k\).  Sampling six vertices and then
an injective image of the relevant graph gives
\begin{align}
 p&=\sum_{H\in\mathcal H_6}\frac{e(H)}{15}p_6(H;W),\notag\\
 \homd(C_4,W)&=\sum_{H\in\mathcal H_6}\frac{N_4(H)}{45}p_6(H;W),\notag\\
 \homd(C_6,W)&=\sum_{H\in\mathcal H_6}\frac{N_6(H)}{60}p_6(H;W).
 \label{eq:six-point-dictionary}
\end{align}
The denominators are, respectively, \(\binom62=15\),
\(6\cdot5\cdot4\cdot3/8=45\), and \(6!/12=60\).

A type is a labeled graph.  A flag of that type is a graph containing an
induced copy of the type on its labeled vertices, with isomorphisms required
to preserve each label.  We use the empty type with three-vertex flags and the nonedge and edge types
on two labeled vertices with four-vertex flags.  Denote the latter two
types by $\sigma_{\rm ne}$ and $\sigma_{\rm e}$, respectively.  For each
\(\sigma\in\{\varnothing,\sigma_{\rm ne},\sigma_{\rm e}\}\), let
\(M_H^\sigma\) be its downward product matrix on \(H\).  The supplementary
certificate, encoded and verified as described in
\cref{app:certificate}, gives rational positive semidefinite matrices
\(Q_\sigma\) and rational numbers \(\eta_H\ge0\) such
that, for every \(H\in\mathcal H_6\),
\begin{equation}\label{eq:six-point-certificate}
 \frac{119}{100}\frac{N_4(H)}{45}
 -\frac{38}{75}\frac{N_6(H)}{60}
 -\frac{e(H)}{15}+\frac{881}{2400}
 =
 \sum_{\sigma\in\{\varnothing,\sigma_{\rm ne},\sigma_{\rm e}\}}
 \langle Q_\sigma,M_H^\sigma\rangle+\eta_H,
\end{equation}
where \(\langle X,Y\rangle=\Tr(X^{\mathsf T}Y)\).
Averaging \eqref{eq:six-point-certificate} with weights \(p_6(H;W)\)
turns each matrix term into a nonnegative quadratic form.  Indeed, if
\(v_\sigma\) is the vector of flag densities conditional on a labeled copy
of $\sigma$, then
\[
 \sum_{H\in\mathcal H_6}p_6(H;W)M_H^\sigma
 =\mathbb E\!\left[
   \mathbf1_{\{\text{the labeled vertices induce }\sigma\}}
   v_\sigma v_\sigma^{\mathsf T}\right]\succeq0,
\]
with the indicator omitted for the empty type; here $\succeq0$ denotes
positive semidefiniteness.  Since also \(\eta_H\ge0\),
averaging gives
\[
 \frac{119}{100}\left(\homd(C_4,W)-\frac18\right)
 -\frac{38}{75}\left(\homd(C_6,W)-\frac1{32}\right)
 -\left(p-\frac12\right)\ge0.
\]
Multiplication by \(2400\) gives \eqref{eq:flag-c246}.
\end{proof}

\subsection{Coefficient conditions for spectral functionals}

The following proposition records the coefficient comparison used in the
main proof.  It separates the scalar minorant from the graphon inequalities
and gives a common stability estimate.  It will be applied with
$h(x)=L(\tau x)$ and $c=\tau^2$.

\begin{proposition}\label{prop:spectral-functional-criterion}
Let $h\colon[-1/2,1/2]\to\mathbb R$ be continuous and satisfy
$h(x)=O(x^2)$ as $x\to0$.  Suppose that there are $a,b\in\mathbb R$ such
that
\begin{equation}\label{eq:abstract-minorant}
 h(x)\ge ax^4+bx^6\qquad (|x|\le1/2)
\end{equation}
and
\begin{equation}\label{eq:abstract-contact}
 h(1/2)=h(-1/2)=\frac{a}{16}+\frac{b}{64}.
\end{equation}
For $c\in\mathbb R$ and a triangle-free graphon $W$, define
\[
 \Phi_{h,c}(W)=c\homd(K_2,W)-\sum_jh(\lambda_j),
\]
where the sum is over the nonzero eigenvalues of $T_W$, counted with
multiplicity.  Set
\begin{equation}\label{eq:abstract-coefficients}
 \alpha=-\frac{75}{38}b,\qquad
 \beta=a+\frac{357}{152}b,\qquad
 \gamma=c-a-\frac38b.
\end{equation}
If $\alpha,\beta,\gamma\ge0$, then
\begin{equation}\label{eq:abstract-extremal-bound}
 \Phi_{h,c}(W)\le\Phi_{h,c}(B_1)
 =\frac c2-\frac a8-\frac b{32}.
\end{equation}
More precisely,
\begin{equation}\label{eq:abstract-edge-deficit}
 \Phi_{h,c}(B_1)-\Phi_{h,c}(W)
 \ge\gamma\eta+\frac34\beta\eta^2,
 \qquad \eta=\frac12-\homd(K_2,W).
\end{equation}
If $\beta+\gamma>0$, then
\begin{equation}\label{eq:abstract-L1-stability}
 d_1(W,B_1)^2\le
 \frac{140\,[\Phi_{h,c}(B_1)-\Phi_{h,c}(W)]}{35\gamma+\beta}.
\end{equation}
Under the condition $\beta+\gamma>0$, $B_1$ is the unique maximizer
up to weak isomorphism.
\end{proposition}

\begin{proof}
Only finitely many eigenvalues lie outside any neighborhood of the origin,
while $h(x)=O(x^2)$ controls the remaining tail because
$\sum_j\lambda_j^2<\infty$.  Thus the spectral sum is absolutely
convergent.  Put $p=\homd(K_2,W)$ and write
\[
 d_4=\homd(C_4,W)-\frac18,\qquad
 d_6=\homd(C_6,W)-\frac1{32},\qquad
 d_p=p-\frac12.
\]
Let $R(x)=h(x)-ax^4-bx^6$.  The nonzero eigenvalues of $B_1$ are
$1/2$ and $-1/2$, so \eqref{eq:abstract-contact} and the trace identities
give
\begin{align}
 \Phi_{h,c}(B_1)-\Phi_{h,c}(W)
 ={}&a d_4+b d_6-cd_p+\sum_jR(\lambda_j).\label{eq:abstract-first-deficit}
\end{align}
The first three terms admit the exact decomposition
\begin{align}
 a d_4+b d_6-cd_p
 ={}&\alpha\left(
 \frac{119}{100}d_4-\frac{38}{75}d_6-d_p\right)\notag\\
 &+\beta(d_4-d_p)+\gamma(-d_p).
 \label{eq:abstract-deficit}
\end{align}
Indeed, comparison of the coefficients of $d_6,d_4,d_p$, in this order,
gives \eqref{eq:abstract-coefficients}.  The first bracket in
\eqref{eq:abstract-deficit} is nonnegative by \cref{thm:flag-c246}, the
second by \cref{lem:four-cycle-stability}, and the third by
\cref{lem:spectral-mantel}.  Every $R(\lambda_j)$ is nonnegative by
\eqref{eq:abstract-minorant}.  This proves
\eqref{eq:abstract-extremal-bound}.  Since $d_4-d_p=\Delta_4(W)$,
\cref{lem:four-cycle-stability} gives \eqref{eq:abstract-edge-deficit}.
Finally, \cref{lem:quantitative-mantel,lem:four-cycle-stability} give
$d_1(W,B_1)^2\le4\eta$ and $d_1(W,B_1)^2\le140\Delta_4(W)$.
Using both nonnegative terms in \eqref{eq:abstract-deficit}, we obtain
\[
 \Phi_{h,c}(B_1)-\Phi_{h,c}(W)
 \ge\gamma\eta+\beta\Delta_4(W)
 \ge\left(\frac\gamma4+\frac\beta{140}\right)d_1(W,B_1)^2.
\]
This proves \eqref{eq:abstract-L1-stability}.  Equality in the extremal
bound then gives cut distance zero from $B_1$.
\end{proof}

\section{Proofs of the triangle-free extremal and stability results}
\label{sec:triangle-free-problem}

\begin{proof}[Proof of \Cref{thm:global}]
When $\tau=0$, both sides of \eqref{eq:global-sharp} vanish.  Fix
$0<\tau\le\tau_{\mathrm c}$ and apply
\cref{prop:spectral-functional-criterion} with
\[
 h(x)=L(\tau x),\qquad c=\tau^2,\qquad
 a=a_\tau,\qquad b=b_\tau.
\]
The minorant is given by \cref{lem:hermite}, and the contact identity
follows from \eqref{eq:hermite-coefficients}.
By \cref{lem:coefficient-signs}, $\alpha_\tau,\beta_\tau>0$ and
$\gamma_\tau\ge0$ throughout this interval.  Thus
\cref{prop:spectral-functional-criterion} gives both maximality and
uniqueness, including at the endpoint where $\gamma_\tau=0$.
Substitution of $q=1$ in \eqref{eq:bq} gives
\eqref{eq:global-sharp}.

It remains to prove sharpness of the interval.  For fixed $\tau>0$, put
$f_\tau(q)=\Phi_\tau(B_q)$.  By \cref{prop:bq,lem:critical-sign}, if
$\tau>\tau_{\mathrm c}$, then
\[
 f_\tau'(1)=\frac{\tau}{2}
 \left(4\sin\frac\tau2-\tau\right)<0.
\]
Continuity of $f_\tau'$ gives an $\varepsilon>0$ such that it is negative on
$[1-\varepsilon,1]$.  Hence
$f_\tau(1)-f_\tau(1-\varepsilon)
=\int_{1-\varepsilon}^1f_\tau'(q)\,dq<0$, so $B_1$ is not a maximizer.
\end{proof}

For a triangle-free graphon $W$, the proof of \cref{thm:stability}
also gives the explicit bound
\begin{equation}\label{eq:combined-stability}
 d_1(W,B_1)^2\le
 \frac{140\,[\Phi_\tau(B_1)-\Phi_\tau(W)]}
 {35\gamma_\tau+\beta_\tau},
 \qquad 0<\tau\le\tau_{\mathrm c},
\end{equation}
where $\beta_\tau>0$ is defined in \eqref{eq:coefficient-definitions}.
At the endpoint $\gamma_{\tau_{\mathrm c}}=0$, so no positive linear edge
density deficit bound is possible.  The four-vertex inequality supplies
the remaining control in \eqref{eq:combined-stability}.

If $I\subset(0,\tau_{\mathrm c}]$ is nonempty and compact and
$\varepsilon=\Phi_\tau(B_1)-\Phi_\tau(W)$, then
\begin{equation}\label{eq:cut-stability}
 \delta_\square(W,B_1)\le d_1(W,B_1)\le C_I\sqrt\varepsilon
 \qquad(\tau\in I).
\end{equation}
If $0<\tau<\tau_{\mathrm c}$, then also
\begin{equation}\label{eq:edge-deficit-stability}
 0\le\frac12-\homd(K_2,W)\le\frac{\varepsilon}{\gamma_\tau},
\end{equation}
and the coefficient $\gamma_\tau$ cannot be increased.

\begin{proof}[Proof of \Cref{thm:stability}]
The general estimate \eqref{eq:abstract-L1-stability}, applied with
$h(x)=L(\tau x)$ and $c=\tau^2$, gives
\eqref{eq:combined-stability}.  Put
\[
 k(\tau)=\frac{35\gamma_\tau+\beta_\tau}{140\tau^2}
 \quad(0<\tau\le\tau_{\mathrm c}).
\]
By \cref{lem:coefficient-signs}, this function is positive and continuous.
The Taylor expansions at zero give
$\gamma_\tau/\tau^2\to1$ and $\beta_\tau/\tau^2\to0$.
Thus $k$ extends continuously to $k(0)=1/4$, and
$k_*:=\min_{[0,\tau_{\mathrm c}]}k>0$.
Dividing the functional deficit by $\tau^2$ proves
\eqref{eq:normalized-L1-stability} for positive $\tau$, with
$C=k_*^{-1/2}$.  At zero the same bound follows from
$d_1(W,B_1)^2\le4\eta$ in \cref{lem:quantitative-mantel} and $k_*\le1/4$.
Equation \eqref{eq:cut-stability} follows because
$\delta_\square\le d_1$ and $\min I>0$.

The deficit decomposition also gives
\[
 \varepsilon\ge\gamma_\tau\eta+\frac34\beta_\tau\eta^2,
 \qquad \eta=\frac12-\homd(K_2,W),
\]
which proves \eqref{eq:edge-deficit-stability} below the endpoint.
For sharpness of its coefficient, put $f_\tau(q)=\Phi_\tau(B_q)$.
Since $\homd(K_2,B_q)=q/2$,
\[
 \lim_{q\uparrow1}
 \frac{\Phi_\tau(B_1)-\Phi_\tau(B_q)}{1/2-\homd(K_2,B_q)}
 =2f_\tau'(1)=\gamma_\tau.
\]
At the endpoint this limit is zero, excluding a positive linear
coefficient.  The quadratic estimate remains valid:
\[
 \eta\le2\sqrt{\frac{\varepsilon}{3\beta_{\tau_{\mathrm c}}}}.
\]
Its order is sharp, since
\[
 \Phi_{\tau_{\mathrm c}}(B_1)-\Phi_{\tau_{\mathrm c}}(B_q)
 =\frac{\tau_{\mathrm c}^2}{2}
 \left(1-\cos\frac{\tau_{\mathrm c}}2\right)(1-q)^2
 +O((1-q)^3).
\]
The optimality of the distance exponent is proved below.
\end{proof}

\begin{remark}\label{rem:sqrt-sharpness}
Let $W_s$ be the complete bipartite graphon with part measures
$1/2+s$ and $1/2-s$, where $0<s<1/2$.  Its edge deficit is $2s^2$.
For every relabeling $B_1^\varphi$, taking one cut set to be the smaller
part of $W_s$ and the other to be $[0,1]$ gives
\[
 \lVert W_s-B_1^\varphi\rVert_\square
 \ge s\left(\frac12-s\right),
\]
because the degree of $W_s$ is $1/2+s$ on that part, whereas every
relabeling of $B_1$ has degree $1/2$.  Together with
\cref{lem:quantitative-mantel}, this shows that both distances are
$\Theta(s)$.  The nonzero eigenvalues of $T_{W_s}$ are
$\pm\sqrt{1/4-s^2}$, so, uniformly for $0\le\tau\le\tau_{\mathrm c}$,
\[
 \Psi_\tau(B_1)-\Psi_\tau(W_s)
 =\frac{4\sin(\tau/2)}{\tau}s^2+O(s^4),
\]
where the coefficient at zero is interpreted as $2$.
It is positive throughout this interval.  Thus the exponent $1/2$ in
\eqref{eq:normalized-L1-stability} cannot be increased, even for cut distance.
\end{remark}

\begin{theorem}\label{thm:finite-dense}
For $n\ge2$, put
\[
 m_n(\tau)=\max_{G\in\cT_n}\Phi_n(G,\tau).
\]
For every $T>0$ and $n\ge2$,
\begin{equation}\label{eq:finite-variational-limit}
 \sup_{|\tau|\le T}|m_n(\tau)-M_\triangle(\tau)|
 \le\frac{T^2\cosh T+T^4/16}{n-1}.
\end{equation}
In particular,
\begin{equation}\label{eq:finite-global-rate}
 \sup_{0\le\tau\le\tau_{\mathrm c}}
 \left|m_n(\tau)-4\left(1-\cos\frac\tau2\right)\right|
 =O(n^{-1}).
\end{equation}
\end{theorem}

\begin{proof}
Fix $T>0$ and $|\tau|\le T$.  The spectral representation gives
$0\le\Phi_\tau(W)\le\tau^2\homd(K_2,W)$ for every graphon $W$.
In particular, $M_\triangle(\tau)\le\tau^2/2$.
For every $G\in\cT_n$, the exact bridge and $D_G(\tau)\ge0$ yield
\[
 \Phi_n(G,\tau)\le\frac{n}{n-1}\Phi_\tau(W_G)
 \le\frac{n}{n-1}M_\triangle(\tau).
\]
Hence
\begin{equation}\label{eq:finite-upper-rate}
 m_n(\tau)-M_\triangle(\tau)\le\frac{T^2}{2(n-1)}.
\end{equation}

For the reverse bound, let $W$ attain $M_\triangle(\tau)$ and sample a
random graph $H_n$ from $W$: choose $n$ independent uniform points in
$[0,1]$ and, conditional on these points, insert each edge independently
with the probability prescribed by $W$.  Since
$\mathbb E[\#K_3(H_n)]=\binom n3\homd(K_3,W)=0$, the graph $H_n$ is
triangle-free almost surely.
For a fixed simple graph $F$ on $k$ vertices, average over all maps from
$V(F)$ to $[n]$.  On injective maps the expected edge product equals
$\homd(F,W)$; on the remaining maps both the edge product and
$\homd(F,W)$ lie in $[0,1]$.  The probability of a collision is at most
$\binom{k}{2}/n$, so
\begin{equation}\label{eq:sampling-collision}
 \left|\mathbb E\homd(F,H_n)-\homd(F,W)\right|
 \le\frac{\binom{k}{2}}n.
\end{equation}
For $F=K_2$ the difference is exactly $-\homd(K_2,W)/n$.
Applying \eqref{eq:sampling-collision} termwise to
\eqref{eq:Phi-series}, whose absolute convergence is uniform on
$[-T,T]$, gives
\begin{align*}
 \left|\mathbb E\Phi_\tau(W_{H_n})-\Phi_\tau(W)\right|
 &\le\frac1n\left[T^2+
 2\sum_{r=2}^{\infty}\frac{T^{2r}}{(2r)!}\binom{2r}{2}\right]\\
 &=\frac{T^2\cosh T}{n}.
\end{align*}
By the triangle-free bound in \cref{lem:defect} and
\eqref{eq:finite-graphon-bridge},
\[
 m_n(\tau)\ge\mathbb E\Phi_n(H_n,\tau)
 \ge\frac{n}{n-1}\left(
 M_\triangle(\tau)-\frac{T^2\cosh T+T^4/16}{n}\right).
\]
Since $M_\triangle(\tau)\ge0$, this proves
\[
 M_\triangle(\tau)-m_n(\tau)
 \le\frac{T^2\cosh T+T^4/16}{n-1}.
\]
Together with \eqref{eq:finite-upper-rate}, this gives
\eqref{eq:finite-variational-limit}.  The last assertion follows from
\cref{thm:global}.
\end{proof}

\begin{corollary}\label{cor:finite-stability}
Let $I\subset(0,\tau_{\mathrm c}]$ be a nonempty compact interval,
let $T=\max I$, and set
\[
 \omega_I=\min_{\tau\in I}(35\gamma_\tau+\beta_\tau),
 \qquad K_T=T^2(1+\cosh T)+\frac5{16}T^4.
\]
For $G\in\cT_n$, $n\ge2$, $\tau\in I$, and
$\varepsilon=m_n(\tau)-\Phi_n(G,\tau)$,
\begin{equation}\label{eq:finite-quantitative-stability}
 d_1(W_G,B_1)
 \le\sqrt{\frac{140}{\omega_I}
 \left(\varepsilon+\frac{K_T}{n-1}\right)}.
\end{equation}
Consequently, if $\tau_n\in I$, $G_n\in\cT_n$, and
$m_n(\tau_n)-\Phi_n(G_n,\tau_n)\to0$, then
\[
 d_1(W_{G_n},B_1)\longrightarrow0,
 \qquad \frac{2e(G_n)}{n^2}\longrightarrow\frac12.
\]
\end{corollary}

\begin{proof}
The finite approximation theorem and \eqref{eq:bridge-error} give
\begin{align*}
 0\le\Phi_\tau(B_1)-\Phi_\tau(W_G)
 &\le|M_\triangle(\tau)-m_n(\tau)|
 +\varepsilon+|\Phi_n(G,\tau)-\Phi_\tau(W_G)|\\
 &\le\varepsilon+\frac{K_T}{n-1}.
\end{align*}
Apply \eqref{eq:combined-stability} to obtain
\eqref{eq:finite-quantitative-stability}.  The $L^1$ convergence assertion
follows at once, and edge density convergence follows from cut continuity
of $\homd(K_2,\cdot)$.
\end{proof}

\section{The bipartite graphon problem}\label{sec:bipartite-problem}

\subsection{Fourth moments and general spectral functionals}

For a bipartite graphon $W$, choose a partition $[0,1]=X\sqcup Y$ on
whose two diagonal blocks $W$ vanishes.  Write $a=\mu(X)$, $b=\mu(Y)$,
and $p=\homd(K_2,W)$.  When $p>0$, both part measures are positive.
The cross operator
\[
 \mathcal K\colon L^2(Y)\longrightarrow L^2(X),\qquad
 (\mathcal Kf)(x)=\int_YW(x,y)f(y)\,dy
\]
has singular values $s_1\ge s_2\ge\cdots\ge0$; the nonzero eigenvalues
of $T_W$ are $\pm s_i$ for $s_i>0$. When $p>0$, put
$u=\mathbf1_X/\sqrt a$ and $v=\mathbf1_Y/\sqrt b$. Then
\begin{equation}\label{eq:bip-singular-bound}
 s_1\ge m:=\langle u,\mathcal Kv\rangle
 =\frac{p}{2\sqrt{ab}}\ge p,
 \qquad p\le2ab\le\frac12.
\end{equation}

The next lemma quantifies the rank-one reduction at fixed edge density.
It uses the relation between fourth moments and quasirandomness
\cite{ChungGrahamWilson1989,LovaszSos2008}, in a form adapted to varying
bipartition sizes.

\begin{lemma}\label{lem:bip-fourth-moment}
Every bipartite graphon $W$ of edge density $p$ satisfies
\begin{equation}\label{eq:bip-fourth-stability}
 \delta_\square(W,B_{2p})\le5D_4^{1/4},
 \qquad D_4=\sum_i s_i^4-p^4=\frac12\homd(C_4,W)-p^4\ge0.
\end{equation}
The constant is independent of $p$.
\end{lemma}

\begin{proof}
The case $p=0$ is immediate, so suppose $p>0$.
Relative to $L^2(X)=\mathbb Ru\oplus u^\perp$ and
$L^2(Y)=\mathbb Rv\oplus v^\perp$, write
\[
 \mathcal K=\begin{pmatrix}m&z^*\\ w&A\end{pmatrix}.
\]
Block multiplication gives
\begin{align*}
 \Tr\bigl((\mathcal K^*\mathcal K)^2\bigr)
 ={}&(m^2+\lVert w\rVert_2^2)^2
 +2\lVert mz+A^*w\rVert_2^2\\
 &+\lVert z\rVert_2^4+2\lVert Az\rVert_2^2
 +\Tr\bigl((A^*A)^2\bigr).
\end{align*}
Since $m\ge p$, this implies
\[
 D_4\ge\lVert w\rVert_2^4+\lVert z\rVert_2^4
             +\Tr\bigl((A^*A)^2\bigr).
\]
Here $u\otimes v$ denotes the rank-one operator $f\mapsto\langle v,f\rangle u$.
Splitting the error into its off-diagonal blocks and its lower-right block,
we obtain
\[
 \opnorm{\mathcal K-mu\otimes v}
 \le\max\{\lVert w\rVert_2,\lVert z\rVert_2\}+\opnorm A
 \le2D_4^{1/4}.
\]
Let $s=|a-1/2|$, put $r=p/(2ab)=2p/(1-4s^2)\le1$, and let $W_0$
be constant with value $r$ across $(X,Y)$ and zero within the parts.
The preceding operator estimate also bounds
$\lVert W-W_0\rVert_\square$.  Balancing the parts changes $W_0$ in
$L^1$ by $2rs(1-s)$, and changing the resulting weight $r$ to $2p$
costs $2rs^2$.  Also, both $W_0$ and $B_{2p}$ have integral $p$.
Consequently,
\begin{equation}\label{eq:bip-balancing-cost}
 \delta_\square(W_0,B_{2p})\le\min\{2rs,2p\}.
\end{equation}
Equation \eqref{eq:bip-singular-bound} gives
\[
 D_4\ge m^4-p^4
 =p^4\bigl((1-4s^2)^{-2}-1\bigr).
\]
If $s\le1/4$, then $r\le8p/3$ and
$D_4\ge8p^4s^2$, so $2rs\le2D_4^{1/4}$.
If $s\ge1/4$, then $D_4\ge(7/9)p^4$, so
$2p\le2(9/7)^{1/4}D_4^{1/4}<3D_4^{1/4}$.
Together with the operator estimate, this proves
\eqref{eq:bip-fourth-stability}.

\end{proof}

\begin{proposition}\label{prop:bip-reduction}
Let $h\colon[-1/2,1/2]\to\mathbb R$ be even and continuous, with
$h(0)=0$ and $h(x)=O(x^2)$ as $x\to0$.  Suppose that $h$ is strictly
increasing on $[0,1/2]$.  For $c\in\mathbb R$, set
\[
 \Phi_{h,c}(W)=c\homd(K_2,W)-\sum_jh(\lambda_j),
\]
where the nonzero eigenvalues of $T_W$ are counted with multiplicity.
Then every bipartite graphon $W$ of edge density $p$ satisfies
\begin{equation}\label{eq:bip-reduction}
 \Phi_{h,c}(W)\le cp-2h(p)=\Phi_{h,c}(B_{2p}),
\end{equation}
with equality if and only if $W$ is weakly isomorphic to $B_{2p}$.
If, for some $\kappa>0$, the function
\begin{equation}\label{eq:bip-h-derivative}
 s\longmapsto h(s)-\kappa s^4
 \quad\text{is nondecreasing on }[0,1/2],
\end{equation}
then
\begin{equation}\label{eq:bip-general-stability}
 \delta_\square(W,B_{2p})
 \le5\left(
 \frac{\Phi_{h,c}(B_{2p})-\Phi_{h,c}(W)}{2\kappa}
 \right)^{1/4}.
\end{equation}
In particular, no positive lower bound on $p$ is required.
\end{proposition}

\begin{proof}
The spectral sum converges absolutely by the assumption at zero and the
Hilbert--Schmidt property.  If $p=0$, then $W=0$ almost everywhere and
the first assertion is immediate.  For $p>0$, evenness of $h$ gives
\[
 \Phi_{h,c}(W)=cp-2\sum_i h(s_i).
\]
Since $h(0)=0$ and $h$ is strictly increasing, \eqref{eq:bip-singular-bound}
implies $\sum_i h(s_i)\ge h(s_1)\ge h(p)$.  Equality forces
$s_i=0$ for $i\ge2$ and $s_1=p$.  Equality throughout
$s_1\ge p/(2\sqrt{ab})\ge p$ then gives $a=b=1/2$ and equality in
the singular value bound at $u,v$.  Thus $\mathcal K=p\,u\otimes v$,
whose kernel is the constant $2p$.  The converse is direct.

Under \eqref{eq:bip-h-derivative}, the function
$h(s)-\kappa s^4$ is nonnegative and nondecreasing.  Hence
\begin{align*}
 \Phi_{h,c}(B_{2p})-\Phi_{h,c}(W)
 &=2\left(h(s_1)-h(p)+\sum_{i\ge2}h(s_i)\right)\\
 &\ge2\kappa\left(\sum_i s_i^4-p^4\right).
\end{align*}
Apply \cref{lem:bip-fourth-moment} to conclude.
\end{proof}

For quantum transport with $\tau>0$, take $h_\tau(s)=L(\tau s)$ and $c=\tau^2$.
Then
\[
 h_\tau'(s)=2\tau(\tau s-\sin(\tau s))>0\qquad(s>0).
\]
For every nonempty compact interval $I\subset(0,\infty)$, the constant
\begin{equation}\label{eq:bip-kappa}
 \kappa_I=\min_{\substack{\tau\in I\\0\le s\le1/2}}
 \frac{h_\tau'(s)}{4s^3}
\end{equation}
is positive, with the quotient at zero defined by its limit $\tau^4/12$.
The derivative bound $h_\tau'(s)\ge4\kappa_I s^3$ verifies
\eqref{eq:bip-h-derivative}.
Thus \cref{prop:bip-reduction} applies uniformly on $I$, both to the
extremal bound and to its fixed-density stability estimate.

\subsection{The maximizing density and stability}

\begin{theorem}\label{thm:threshold}
For $0\le\tau\le\tau_{\mathrm c}$,
\[
 \sup\{\Phi_\tau(W):W\text{ is a bipartite graphon}\}=\Phi_\tau(B_1),
\]
and $B_1$ is the unique maximizer, up to weak isomorphism, when
$0<\tau\le\tau_{\mathrm c}$.  If $\tau>\tau_{\mathrm c}$, then the unique
bipartite maximizer, up to weak isomorphism, is $B_{q_\tau}$,
where $q_\tau$ is defined by \eqref{eq:q-equation}.
\end{theorem}

\begin{proof}
At $\tau=0$ the functional vanishes.  For $\tau>0$,
\cref{prop:bip-reduction}, with $h=h_\tau$ and $c=\tau^2$, reduces the
problem to maximizing $f_\tau(q)=\Phi_\tau(B_q)$ over $0\le q\le1$.
By \eqref{eq:bq},
\begin{align*}
 f_\tau'(q)&=\frac{\tau^2}{2}(1-2q)+2\tau\sin(\tau q/2),\\
 f_\tau''(q)&=\tau^2(\cos(\tau q/2)-1)\le0.
\end{align*}
The second derivative does not vanish identically on any nondegenerate
interval, so $f_\tau'$ is strictly decreasing and $f_\tau$ is strictly
concave.  Since
\[
 f_\tau'(1)=\tau\left(2\sin(\tau/2)-\frac\tau2\right),
\]
\cref{lem:critical-sign} shows that the unique maximum is at $q=1$ for
$0<\tau\le\tau_{\mathrm c}$.  For $\tau>\tau_{\mathrm c}$,
$f_\tau'(0)=\tau^2/2>0>f_\tau'(1)$, so there is a unique interior
maximizer, characterized by \eqref{eq:q-equation}.
The equality statement follows from \cref{prop:bip-reduction}.
\end{proof}

We next control changes in the maximizing density.  Strict concavity alone
does not give a uniform quadratic bound, because the second derivative
can vanish at the maximizer.

\begin{lemma}\label{lem:bip-scalar-gap}
For each nonempty compact interval $I\subset(0,\infty)$, there is $c_I>0$ such that
\begin{equation}\label{eq:bip-scalar-gap}
 f_\tau(q_*(\tau))-f_\tau(q)
 \ge c_I|q-q_*(\tau)|^4
 \qquad(\tau\in I,\ 0\le q\le1).
\end{equation}
\end{lemma}

\begin{proof}
Put $q_*=q_*(\tau)$ and $t=|q-q_*|$.  The first-order term at the
maximizer has the required sign.  The integral remainder formula gives
\[
 f_\tau(q_*)-f_\tau(q)
 \ge\tau^2t^2\int_0^1(1-v)
 [1-\cos(\alpha\pm\ell v)]\,dv,
 \quad \alpha=\frac{\tau q_*}{2},\quad \ell=\frac{\tau t}{2},
\]
with the sign chosen according to $q-q_*$.  The minimum of the integral
over $\alpha$ is
\[
 J(\ell)=\frac12-
 \left|\int_0^1(1-v)\e^{\ii\ell v}\,dv\right|.
\]
For $\ell>0$, strictness in the triangle inequality gives $J(\ell)>0$.
Expansion at zero yields $J(\ell)=\ell^2/72+O(\ell^4)$.
Therefore, for every $L>0$, the quotient $J(\ell)/\ell^2$ has a positive
continuous extension on $[0,L]$ and a positive minimum $c_L$.
Taking $L=\max I/2$ gives
\[
 f_\tau(q_*)-f_\tau(q)\ge\frac{c_L\tau^4}{4}t^4,
\]
which proves the lemma since $\min I>0$.
\end{proof}

For example, when $\tau=8\pi k$ with an integer $k\ge1$, one has
$q_*(\tau)=1/2$ and
\[
 f_\tau(1/2)-f_\tau(q)
 =\frac{\tau^4}{96}(q-1/2)^4+O((q-1/2)^6).
\]
Thus the quartic order in \cref{lem:bip-scalar-gap} is necessary for
estimates valid on arbitrary compact time intervals.

\begin{theorem}\label{thm:bip-stability}
Let $I\subset(0,\infty)$ be a nonempty compact interval.  There is a
constant $C_I>0$ such that every bipartite graphon $W$ satisfies
\begin{equation}\label{eq:bip-stability}
 \delta_\square(W,B_{q_*(\tau)})
 \le C_I\bigl[\Phi_\tau(B_{q_*(\tau)})-\Phi_\tau(W)\bigr]^{1/4}
 \qquad(\tau\in I).
\end{equation}
For every fixed $\tau>\tau_{\mathrm c}$, the exponent $1/4$ is best
possible.  For $0<\tau\le\tau_{\mathrm c}$, the stronger estimate with
optimal exponent $1/2$ is given by \cref{thm:stability}.
\end{theorem}

\begin{proof}
Fix $I$ as in the statement, put $p=\homd(K_2,W)$, and write
\[
 \varepsilon=f_\tau(q_*(\tau))-\Phi_\tau(W)
 =[f_\tau(q_*(\tau))-f_\tau(2p)]+E_p,
 \qquad E_p=f_\tau(2p)-\Phi_\tau(W).
\]
Both terms are nonnegative by \cref{prop:bip-reduction}.
\Cref{lem:bip-scalar-gap} gives
$|2p-q_*(\tau)|\le C_I\varepsilon^{1/4}$.
Equations \eqref{eq:bip-general-stability} and \eqref{eq:bip-kappa}
now give, for every $p\in[0,1/2]$,
\[
 \delta_\square(W,B_{2p})
 \le5\left(\frac{E_p}{2\kappa_I}\right)^{1/4}
 \le C_I\varepsilon^{1/4}.
\]
Since $\delta_\square(B_{2p},B_{q_*})\le|2p-q_*|/2$, the triangle
inequality proves \eqref{eq:bip-stability}.

For sharpness, fix $\tau>\tau_{\mathrm c}$ and put $q=q_\tau\in(0,1)$.
Split each of the two equal parts $X,Y$ into equal halves, and let
$a\colon X\to\{\pm1\}$ and $b\colon Y\to\{\pm1\}$ take the two
values on these halves.  For $0<\delta<\min(q,1-q)$, define
$W_\delta=q+\delta a(x)b(y)$ on $X\times Y$, extend symmetrically, and
set it to zero on the two diagonal blocks.  The cross singular values
are $q/2$ and $\delta/2$, and the edge density remains $q/2$.  Thus
\begin{equation}\label{eq:bip-sharp-deficit}
 \Phi_\tau(B_q)-\Phi_\tau(W_\delta)
 =2L(\tau\delta/2)=\frac{\tau^4}{96}\delta^4+O(\delta^6).
\end{equation}
To bound the distance after arbitrary relabeling, let $U$ be $B_q$ on a
balanced partition $C,D$, and put $d=\lVert W_\delta-U\rVert_\square$.
After interchanging $C,D$, we may assume
$v=\mu(X\cap D)=\mu(Y\cap C)\le1/4$.  Taking the cut $X\times X$ gives
\[
 d\ge2qv(1/2-v)\ge qv/2.
\]
Extend $a,b$ by zero outside their respective parts.  Their zero means
give
\[
 \int aW_\delta b=\frac\delta4,
 \qquad
 \left|\int aUb\right|\le2qv^2\le\frac{8d^2}{q}.
\]
Expanding the signs into four rectangles gives
$|\int a(W_\delta-U)b|\le4d$.  Hence
$\delta/4\le4d+8d^2/q$.
If $d\le q/2$, this implies $d\ge\delta/32$; otherwise
$d>q/2>\delta/2$.  Taking the infimum over balanced partitions gives the
following bounds, which, together with \eqref{eq:bip-sharp-deficit},
exclude every exponent greater than $1/4$ at this fixed $\tau$:
\[
 \frac\delta{32}\le\delta_\square(W_\delta,B_q)
 \le\lVert W_\delta-B_q\rVert_1=\frac\delta2.
\]
\end{proof}

\section{The global maximizer beyond the critical time}
\label{sec:second-phase}

The extension beyond $\tau_{\mathrm c}$ uses stationarity under changes
of the vertex measure.  Unlike a variation of edge weights, this variation
preserves triangle-freeness automatically.  Critical stability first puts
all nearby maximizers close to $B_1$.  The stationarity identity then
excludes vertices of small degree.

\subsection{Variation of the vertex measure}

For a bounded self-adjoint operator $S$, define
\begin{equation}\label{eq:vertex-operator}
 \mathcal Q_\tau(S)
 =2\sum_{r=2}^\infty
   \frac{(-1)^r\tau^{2r}}{(2r-1)!}S^{2r-2}.
\end{equation}
The series converges in operator norm, locally uniformly in $\tau$ and
in bounded sets of operators.  In scalar notation its defining function is
$h_\tau'(s)/s$, with value zero at $s=0$, where $h_\tau(s)=L(\tau s)$.
For a graphon $W$, write $g_x=W(x,\cdot)$ and
$d_W(x)=\int g_x$.  Set
\begin{equation}\label{eq:vertex-potential}
 \zeta_{\tau,W}(x)
 =2\tau^2d_W(x)
  -\langle g_x,\mathcal Q_\tau(T_W)g_x\rangle.
\end{equation}
This is a bounded measurable function, defined almost everywhere.  Write
$\overline\zeta_{\tau,W}=\int\zeta_{\tau,W}$; integrating
\eqref{eq:vertex-potential} gives \eqref{eq:mean-vertex-potential} below
for every graphon.

\begin{lemma}\label{lem:vertex-stationarity}
If $W$ maximizes $\Phi_\tau$ over the triangle-free graphons, then
$\zeta_{\tau,W}$ is constant almost everywhere.  Its constant value is
\begin{equation}\label{eq:mean-vertex-potential}
 \overline\zeta_{\tau,W}
 =2\tau^2\homd(K_2,W)
  -2\sum_{r=2}^\infty
       \frac{(-1)^r\tau^{2r}}{(2r-1)!}\homd(C_{2r},W).
\end{equation}
\end{lemma}

\begin{proof}
Let $f$ be any bounded real measurable function with $\int f=0$.
For $|t|\lVert f\rVert_\infty<1$, replace Lebesgue measure $\mu$ by
$d\mu_t=(1+tf)\,d\mu$, leaving the kernel $W$ unchanged.
This is an atomless probability space and hence gives a graphon on
$[0,1]$ after a measure-space isomorphism.  Triangle-freeness is preserved,
since $\mu_t$ and $\mu$ are equivalent.

Differentiating the homomorphism density of $K_2$ at $t=0$ gives
$2\int f(x)d_W(x)\,dx$.  For $C_{2r}$, the derivative is
\[
 2r\int f(x)\langle g_x,T_W^{2r-2}g_x\rangle\,dx.
\]
The factorial decay of the coefficients in \eqref{eq:Phi-series}
justifies differentiating the sum.  Thus the first variation is
$\int f\zeta_{\tau,W}$.  Both signs of $t$ are admissible, so maximality
makes this integral zero for every such $f$.  Therefore
$\zeta_{\tau,W}$ is constant almost everywhere.  Integrating
\eqref{eq:vertex-potential} gives \eqref{eq:mean-vertex-potential}.
\end{proof}

\subsection{From stationarity to bipartiteness}

We record the graphon version of the triangle-free case of the
Andr\'asfai--Erd\H{o}s--S\'os theorem
\cite[Lemma~1.2]{AndrasfaiErdosSos1974}.

\begin{lemma}\label{lem:graphon-aes}
If a triangle-free graphon $W$ satisfies
$\operatorname*{ess\,inf}_x d_W(x)>2/5$, then $W$ is bipartite.
\end{lemma}

\begin{proof}
Choose $d_0>2/5$ with $d_W(x)\ge d_0$ almost everywhere, and sample
$G(n,W)$.  Conditional on the sampled point $x_i$, the degree of vertex
$i$ is a sum of $n-1$ independent Bernoulli variables of mean $d_W(x_i)$.
A concentration bound and a union bound therefore show that
$\delta(G(n,W))>2n/5$ with probability tending to one.  The sampled graph
is triangle-free almost surely.  The finite theorem makes it bipartite
with probability tending to one.  Since $G(n,W)$ converges to $W$ in cut
distance \cite{Lovasz2012}, we may choose bipartite graphons $U_n$
converging to $W$ in cut distance.

For completeness, bipartiteness is closed under this convergence.
Relabel so that $\lVert U_n-W\rVert_\square\to0$, and let $f_n$ be the
indicator of one part of a bipartition of $U_n$.  A subsequence converges
weakly in $L^2$ to a function $f$ with $0\le f\le1$.  Compactness of $T_W$
and the cut norm bound give
\[
 \langle f,T_Wf\rangle
 =\langle 1-f,T_W(1-f)\rangle=0.
\]
Nonnegativity implies that $W(x,y)>0$ only where
$\{f(x),f(y)\}=\{0,1\}$, up to a null set.  Thus
$\{f=0\}$ and its complement form a bipartition of $W$.
\end{proof}

\begin{lemma}\label{lem:local-bipartiteness}
Fix $0<\theta\le\tau_{\mathrm c}$.  Suppose $\tau_n\to\theta$ and
$W_n$ are triangle-free graphons such that
\[
 d_1(W_n,B_1)\longrightarrow0,
 \qquad
 \lVert\zeta_{\tau_n,W_n}-\overline\zeta_{\tau_n,W_n}\rVert_\infty
 \longrightarrow0.
\]
Then $\lVert d_{W_n}-1/2\rVert_\infty\to0$, where the norms are essential
suprema.  In particular, $W_n$ is bipartite for all sufficiently large $n$.
\end{lemma}

\begin{proof}
Put $T_0=T_{B_1}$, with balanced parts $X,Y$, and let $P$ be the orthogonal
projection onto $\operatorname{span}\{\mathbf1_X,\mathbf1_Y\}$.
Set
\[
 c_\theta=2\theta^2-4\theta\sin(\theta/2).
\]
The inequalities $\sin(\theta/2)<\theta/2$ and
$4\sin(\theta/2)\ge\theta$ give $0<c_\theta\le\theta^2$.
Since the nonzero eigenvalues of $T_0$ are $\pm1/2$,
\begin{equation}\label{eq:critical-vertex-operator}
 \mathcal Q_\theta(T_0)=c_\theta P,
 \qquad
 \overline\zeta_{\theta,B_1}=\theta^2-\frac{c_\theta}{2}.
\end{equation}
At the critical time $c_\theta=\theta^2$.

After relabeling, assume $\lVert W_n-B_1\rVert_1\to0$.
Writing $T_n=T_{W_n}$, we have
\[
 \eta_n:=\opnorm{T_n-T_0}
 \le\lVert W_n-B_1\rVert_2
 \le\lVert W_n-B_1\rVert_1^{1/2}\longrightarrow0.
\]
Uniform convergence of \eqref{eq:vertex-operator} and continuity of
operator powers imply
\[
 e_n:=\opnorm{\mathcal Q_{\tau_n}(T_n)-c_\theta P}\longrightarrow0.
\]
Equation \eqref{eq:mean-vertex-potential} gives
$\overline\zeta_{\tau_n,W_n}\to\theta^2-c_\theta/2$.
Here homomorphism densities are continuous under $L^1$ convergence,
and the series is locally uniformly convergent.

For almost every $x$, let $g=W_n(x,\cdot)$ and put
$u=\int_Xg$, $v=\int_Yg$, and $d=u+v=d_{W_n}(x)$.
Triangle-freeness implies $\langle g,T_ng\rangle=0$.  Since
$\lVert g\rVert_2\le1$, we have
\[
 2uv=\langle g,T_0g\rangle\le\eta_n,
 \qquad \langle g,Pg\rangle=2(u^2+v^2).
\]
Also $u,v\le1/2$, so
$d\le1/2+\sqrt{\eta_n/2}$.
The operator estimate and the assumed near-constancy of the potential
therefore give, uniformly for almost every $x$,
\[
 F_\theta(d_{W_n}(x))=F_\theta(1/2)+o(1),
 \qquad F_\theta(d)=2\theta^2d-2c_\theta d^2.
\]
The function $F_\theta$ is strictly increasing on $[0,1/2]$.  More
explicitly, for $0\le d\le1/2$,
\[
 F_\theta(1/2)-F_\theta(d)
 =(1/2-d)\bigl(2\theta^2-c_\theta(1+2d)\bigr)
 \ge2\theta^2(1/2-d)^2.
\]
Together with the upper bound on $d$, this proves
$\lVert d_{W_n}-1/2\rVert_\infty\to0$.
Eventually $d_{W_n}>9/20$ almost everywhere, and
\cref{lem:graphon-aes} gives bipartiteness.
\end{proof}

\begin{proof}[Proof of \Cref{thm:second-phase}]
By \cref{cor:continuity-attainment}, a maximizing graphon exists for every
$\tau$.  The series \eqref{eq:Phi-series} gives
\[
 \omega(t):=\sup_{\substack{|\tau-\tau_{\mathrm c}|\le t\\
                           W\text{ a graphon}}}
 |\Phi_\tau(W)-\Phi_{\tau_{\mathrm c}}(W)|\longrightarrow0
 \quad\text{as }t\downarrow0.
\]
If $W$ maximizes $\Phi_\tau$ over the triangle-free graphons, then
$\Phi_\tau(W)\ge\Phi_\tau(B_1)$, so
\[
 0\le\Phi_{\tau_{\mathrm c}}(B_1)
       -\Phi_{\tau_{\mathrm c}}(W)
 \le2\omega(|\tau-\tau_{\mathrm c}|).
\]
Critical stability in \cref{thm:stability} therefore implies
$d_1(W,B_1)\to0$ as $\tau\to\tau_{\mathrm c}$, uniformly over all
maximizers.  If nonbipartite maximizers existed at times arbitrarily
close to $\tau_{\mathrm c}$, choose a sequence of them with
$\tau_n\to\tau_{\mathrm c}$.  \Cref{lem:vertex-stationarity} makes
their potentials constant, and \cref{lem:local-bipartiteness}, with
$\theta=\tau_{\mathrm c}$, gives a contradiction.  Thus every maximizer
is bipartite when $|\tau-\tau_{\mathrm c}|$ is sufficiently small.  For
$\tau>\tau_{\mathrm c}$, \cref{thm:threshold} identifies it uniquely as
$B_{q_\tau}$ and gives \eqref{eq:second-phase}.

\end{proof}

\begin{corollary}\label{cor:phase-regularity}
The function $M_\triangle$ is continuously differentiable in a neighborhood of
$\tau_{\mathrm c}$, but is not twice differentiable at $\tau_{\mathrm c}$.  More precisely, as
$\tau\downarrow\tau_{\mathrm c}$,
\begin{equation}\label{eq:second-phase-gain}
 M_\triangle(\tau)-4\left(1-\cos\frac\tau2\right)
 =\frac{(1/2-\cos(\tau_{\mathrm c}/2))^2}
        {2(1-\cos(\tau_{\mathrm c}/2))}
   (\tau-\tau_{\mathrm c})^2
   +O((\tau-\tau_{\mathrm c})^3).
\end{equation}
\end{corollary}

\begin{proof}
Write $\theta=\tau_{\mathrm c}$ and
$c_0=\cos(\theta/2)$.  With $f_\tau(q)=\Phi_\tau(B_q)$, put
\[
 A=\partial_\tau\partial_qf_\theta(1)=\theta(c_0-1/2),
 \qquad B=\partial_q^2f_\theta(1)=\theta^2(c_0-1)<0.
\]
The implicit-function theorem gives an analytic branch of stationary
points with
$q_\tau-1=-(A/B)(\tau-\theta)+O((\tau-\theta)^2)$.
Taylor expansion then yields
\[
 f_\tau(q_\tau)-f_\tau(1)
 =-\frac{A^2}{2B}(\tau-\theta)^2+O((\tau-\theta)^3),
\]
which is \eqref{eq:second-phase-gain}.
Below the threshold $M_\triangle(\tau)=f_\tau(1)$; above it the analytic
branch just obtained gives the maximum.  Their values and first derivatives
agree at $\theta$, whereas the right second derivative exceeds the left
one by $(1/2-c_0)^2/(1-c_0)>0$.
\end{proof}

\section{Exact finite maximizers below the threshold}
\label{sec:finite-exactness}

The finite objective contains the diagonal defect in
\eqref{eq:finite-graphon-bridge}.  Its global $O(n^{-1})$ bound is not
sufficient for comparing graphs that differ by one edge.  We therefore
express it as a graphon functional and estimate its change under cloning
and edge addition.

For $k\ge2$, let
\[
 b_k(x;W)=\int_{[0,1]^{k-1}}
 W(x,x_1)W(x_1,x_2)\cdots W(x_{k-1},x)\,dx_1\cdots dx_{k-1},
\]
where $b_2(x;W)=\int W(x,y)^2\,dy$.  Define
\begin{equation}\label{eq:diagonal-functional}
 \mathcal R_\tau(W)
 =\int_0^1\left|
       \sum_{k=2}^\infty\frac{(-\ii\tau)^k}{k!}b_k(x;W)
             \right|^2\,dx.
\end{equation}
The series converges absolutely and uniformly for bounded $\tau$.

\begin{lemma}\label{lem:defect-variation}
For every simple graph $G$ on $n\ge2$ vertices,
\begin{equation}\label{eq:finite-defect-functional}
 \Phi_n(G,\tau)=\frac{n}{n-1}
 \left(\Phi_\tau(W_G)-\frac1n\mathcal R_\tau(W_G)\right).
\end{equation}
For every $T>0$, there is $C_T$ such that, for $|\tau|\le T$,
\begin{equation}\label{eq:defect-kernel-lipschitz}
 |\mathcal R_\tau(W)-\mathcal R_\tau(V)|
 \le C_T\lVert W-V\rVert_1.
\end{equation}
For a fixed kernel on a common measurable space and probability measures
$\mu,\nu$, one also has
\begin{equation}\label{eq:defect-measure-lipschitz}
 |\mathcal R_\tau(W;\mu)-\mathcal R_\tau(W;\nu)|
 \le C_T\lVert\mu-\nu\rVert_{\mathrm{TV}},
\end{equation}
where the total variation norm is the mass of the variation measure.
\end{lemma}

\begin{proof}
On the interval for vertex $i$, we have
$b_k(x;W_G)=(A^k)_{ii}/n^{k-1}$.  Since $A_{ii}=0$, the sum in
\eqref{eq:diagonal-functional} equals
$n((U_G(\tau/n))_{ii}-1)$.  Thus
\[
 \mathcal R_\tau(W_G)
 =n\sum_i| (\exp(-\ii\tau A/n))_{ii}-1|^2,
\]
and \eqref{eq:finite-defect-functional} follows from
\eqref{eq:finite-graphon-bridge}.

Expanding the squared modulus in \eqref{eq:diagonal-functional} expresses
$\mathcal R_\tau$ as a sum, indexed by $k,l\ge2$, of homomorphism
integrals for two rooted cycles with their roots identified.  A cycle of
length two is interpreted as two parallel edges.  Each resulting
multigraph has $k+l$ edges counted with multiplicity and $k+l-1$ vertices;
the coefficient has modulus $|\tau|^{k+l}/(k!l!)$.
Telescoping over edges bounds the change of its integral by
$(k+l)\lVert W-V\rVert_1$.  Telescoping over vertex measures bounds the
change by $(k+l-1)\lVert\mu-\nu\rVert_{\mathrm{TV}}$.
Summing these bounds proves both estimates, since the factorially weighted
sums converge uniformly for $|\tau|\le T$.
\end{proof}

\begin{lemma}\label{lem:finite-cloning}
For every $T>0$, there is $C_T$ such that, whenever $G\in\cT_n$ maximizes
$\Phi_n(\cdot,\tau)$ and $|\tau|\le T$,
\[
 \lVert\zeta_{\tau,W_G}-\overline\zeta_{\tau,W_G}\rVert_\infty
 \le\frac{C_T}{n}.
\]
\end{lemma}

\begin{proof}
Regard $W_G$ as its adjacency kernel on $n$ atoms of mass $1/n$.
For vertices $i,j$, replacing $i$ by a nonadjacent copy of $j$ preserves
triangle-freeness.  The resulting graphon is the same kernel with the
probability measure changed from $\mu$ to
$\mu' =\mu+(\delta_j-\delta_i)/n$.
The vertex-measure derivative used in \cref{lem:vertex-stationarity}
therefore gives
\[
 \Phi_\tau(W_G;\mu')-\Phi_\tau(W_G;\mu)
 =\frac{\zeta_{\tau,W_G}(j)-\zeta_{\tau,W_G}(i)}{n}
  +O_T(n^{-2}).
\]
The remainder is uniform: for a homomorphism integral on $v$ vertices,
its second derivative along the segment from $\mu$ to $\mu'$ is bounded
by $v(v-1)\lVert\mu'-\mu\rVert_{\mathrm{TV}}^2$, and the corresponding
series for \eqref{eq:Phi-series} is summable.
By \eqref{eq:defect-measure-lipschitz}, the change of
$n^{-1}\mathcal R_\tau$ is $O_T(n^{-2})$.
Maximality and \eqref{eq:finite-defect-functional} consequently imply
$\zeta_{\tau,W_G}(j)-\zeta_{\tau,W_G}(i)\le C_T/n$.
This holds for every ordered pair $i,j$, and averaging proves the claim.
\end{proof}

\begin{proof}[Proof of \Cref{thm:finite-exactness}]
Let $n_k\to\infty$, let $\tau_k\in I$, and let $G_k\in\cT_{n_k}$ be
maximizers.  Pass to a subsequence such that $\tau_k\to\theta\in I$.
By \cref{cor:finite-stability}, applied on the convex hull of $I$,
$d_1(W_{G_k},B_1)\to0$.
\Cref{lem:finite-cloning,lem:local-bipartiteness} give
\[
 \max_{v\in V(G_k)}\left|\frac{d_{G_k}(v)}{n_k}-\frac12\right|
 \longrightarrow0,
\]
and $G_k$ is bipartite for all sufficiently large $k$.
For the last assertion one may apply the finite
Andr\'asfai--Erd\H{o}s--S\'os theorem directly, since eventually
$\delta(G_k)>9n_k/20$.

Choose a bipartition with normalized part sizes $a_k,b_k$.
Since $\homd(K_2,G_k)\to1/2$ and
$\homd(K_2,G_k)\le2a_kb_k$, we have $a_k,b_k\to1/2$.
Let $U_k$ be the complete bipartite graphon on this same partition and put
$W_k=W_{G_k}$.  The uniform degree estimate implies
\begin{equation}\label{eq:finite-row-deficit}
 r_k:=\operatorname*{ess\,sup}_x
       \int|U_k(x,y)-W_k(x,y)|\,dy\longrightarrow0.
\end{equation}

We show that $G_k$ is complete bipartite for all sufficiently large $k$.
Suppose that a cross-edge is missing, and let $W_{k,t}$, $0\le t\le1$,
be obtained by increasing its two symmetric kernel entries from zero to
$t$.  Then $W_k\le W_{k,t}\le U_k$, so \eqref{eq:finite-row-deficit}
also holds with $W_{k,t}$, uniformly in $t$.
Write $H_{\tau,V}$ for the kernel of $h_\tau'(T_V)$.
Its power series starts with $T_V^3$.  For every integer $j\ge2$,
telescoping the kernel of $T_V^j-T_{U_k}^j$ bounds it in absolute value
by $jr_k$ whenever $W_k\le V\le U_k$.  Each telescoping term is bounded
by a row integral in \eqref{eq:finite-row-deficit}, or by its integral
over $x$.  Summing the power series yields
\[
 \lVert H_{\tau_k,W_{k,t}}-H_{\tau_k,U_k}\rVert_\infty
 =o(1)
\]
uniformly in $t$.  The kernel $H_{\tau_k,U_k}$ is constant across the bipartition:
\[
 H_{\tau_k,U_k}(x,y)
 =\frac{h_{\tau_k}'(\sqrt{a_kb_k})}{\sqrt{a_kb_k}}
 \longrightarrow2\theta^2-4\theta\sin(\theta/2).
\]
Differentiating $\Phi_{\tau_k}$ with respect to the missing edge weight
therefore gives
\[
 \frac{d}{dt}\Phi_{\tau_k}(W_{k,t})
 =\frac{2}{n_k^2}
   \bigl(\gamma_\theta+o(1)\bigr),
 \qquad \gamma_\theta=4\theta\sin(\theta/2)-\theta^2>0,
\]
uniformly in $t$.  Indeed, the derivative is
$\tau_k^2\int \dot W_{k,t}
 -\Tr(h_{\tau_k}'(T_{W_{k,t}})T_{\dot W_{k,t}})$.
By \eqref{eq:defect-kernel-lipschitz}, adding the edge changes
$n_k^{-1}\mathcal R_{\tau_k}$ by $O_I(n_k^{-3})$.
Equation \eqref{eq:finite-defect-functional} now shows that adding the
edge strictly increases $\Phi_{n_k}$, a contradiction.

It remains to balance a complete bipartite graph.  For $a+b=n$ and
$x=ab/n^2$, direct calculation gives
\begin{equation}\label{eq:finite-complete-bipartite}
 \Phi_n(K_{a,b},\tau)
 =\frac{n}{n-1}\left[
  4(1-\cos(\tau\sqrt x))
  -\frac{(1-\cos(\tau\sqrt x))^2}{nx}\right].
\end{equation}
The derivative with respect to $x$ of the first term in brackets is
$2\tau\sin(\tau\sqrt x)/\sqrt x$, with limit $2\tau^2$ at zero.
It has a positive minimum on $I\times[0,1/4]$, since
$\max I<\tau_{\mathrm c}<2\pi$.
The function $(1-\cos(\tau\sqrt x))^2/x$ extends analytically to $x=0$
and has uniformly bounded first derivative on this set.
Thus \eqref{eq:finite-complete-bipartite} is strictly increasing in $x$
for all sufficiently large $n$, uniformly for $\tau\in I$.
Its unique maximum occurs at $ab=\lfloor n^2/4\rfloor$.
If the asserted uniform $n_0(I)$ did not exist, a sequence of unbalanced
or noncomplete maximizers would contradict the preceding argument.

Finally, fix $\tau>\tau_{\mathrm c}$.  By \cref{thm:global},
$M_\triangle(\tau)>\Phi_\tau(B_1)$.
\Cref{thm:finite-dense} and
\eqref{eq:finite-complete-bipartite} imply
$m_n(\tau)\to M_\triangle(\tau)$ and
$\Phi_n(T_2(n),\tau)\to\Phi_\tau(B_1)$, respectively.
The strict gap proves the converse.
\end{proof}

\section{Further extremal questions}\label{sec:further-problems}

The bipartite maximizer is characterized for every $\tau>0$, and
\cref{thm:second-phase} identifies it as the global triangle-free maximizer
on an interval immediately beyond $\tau_{\mathrm c}$.  Determining how far
this interval extends remains open.

\begin{problem}\label{prob:second-phase}
Determine the supremum $\tau_{\mathrm b}\in(\tau_{\mathrm c},\infty]$ of
those $T>\tau_{\mathrm c}$ for which $B_{q_\tau}$ is the unique
triangle-free maximizer, up to weak isomorphism, for every
$\tau\in(\tau_{\mathrm c},T)$.
\end{problem}

A basic nonbipartite competitor is the balanced blow-up graphon $W_{C_5}$:
partition $[0,1]$ into five sets of measure $1/5$, put value one between
consecutive parts cyclically, and put value zero elsewhere.  Its nonzero
eigenvalues are $2/5$, $(\sqrt5-1)/10$, and $-(\sqrt5+1)/10$, the last two
with multiplicity two.  Thus
\begin{align}
 \Phi_\tau(W_{C_5})=2\Bigg[&1-\cos\frac{2\tau}{5}
 +2\left(1-\cos\frac{(\sqrt5-1)\tau}{10}\right)\notag\\
 &+2\left(1-\cos\frac{(\sqrt5+1)\tau}{10}\right)\Bigg].
 \label{eq:c5}
\end{align}
Comparing this value with \eqref{eq:bq} provides lower bounds for
$M_\triangle(\tau)$, but cannot determine its maximizing graphon.
The uniform tail estimate in \cref{prop:tail} permits bounds on compact
time intervals to be reduced to finitely many moment inequalities with an
explicit remainder.

The exact finite theorem leaves only the critical endpoint undecided
within the positive-time interval where $B_1$ is globally optimal.

\begin{problem}\label{prob:finite-exactness}
Is $T_2(n)$ the unique maximizer, up to isomorphism, of
$F_G(\tau_{\mathrm c}/n)$ over $G\in\cT_n$ for all sufficiently large $n$?
\end{problem}

The proof of \cref{thm:finite-exactness} uses the strict positivity of
$\gamma_\tau$ to fill missing cross-edges.  At $\tau=\tau_{\mathrm c}$
this first-order gain vanishes, so the diagonal correction and higher-order
terms must be compared more precisely.  For $\tau>\tau_{\mathrm c}$,
another finite problem is to characterize the maximizing bipartite graphs
whose graphons approach the fractional limit $B_{q_\tau}$.

\appendix

\section{Proofs of two technical estimates}\label{app:technical-estimates}

\subsection{Quantitative Mantel stability}

\begin{proof}[Proof of \Cref{lem:quantitative-mantel}]
Set
\[
 d(y)=\int_0^1W(y,z)\,dz,
 \qquad
 h(x)=\int_0^1W(x,y)d(y)\,dy.
\]
Interchanging the nonnegative integrals and applying Cauchy--Schwarz gives
\[
 \int_0^1h(x)\,dx=\int_0^1d(y)^2\,dy\ge p^2.
\]
Since the integrand defining $\homd(K_3,W)$ is nonnegative, we may choose
$x$ such that $h(x)\ge p^2$ and
\[
 \int_{[0,1]^2}W(x,y)W(x,z)W(y,z)\,dy\,dz=0.
\]
After modifying the section $W(x,\cdot)$ on a null set if necessary, put
$X=\{y:W(x,y)>0\}$ and $Y=[0,1]\setminus X$.  The last display implies
that $W=0$ almost everywhere on $X^2$.  If
\[
r=\int_{Y^2}W,
\]
then, since $0\le W(x,y)\le1$,
\[
 \int_Xd(y)\,dy\ge h(x),
 \qquad
 r=p-2\int_Xd(y)\,dy
 \le p-2p^2=2p\eta.
\]
This proves \eqref{eq:mantel-internal}.

Put $s=|\mu(X)-1/2|$.  The capacity of the cut $(X,Y)$ gives
\[
 p-r=2\int_{X\times Y}W
 \le2\mu(X)\mu(Y)=\frac12-2s^2.
\]
Hence
\[
 s^2\le\frac{\eta+r}{2}
 \le\eta(1-\eta),
\]
where the last inequality uses $r\le2p\eta=(1-2\eta)\eta$.  Moreover,
\begin{align*}
 \lVert W-K_{X,Y}\rVert_1
 &=r+2\mu(X)\mu(Y)-(p-r)\\
 &=\eta+2r-2s^2\le3\eta,
\end{align*}
which proves \eqref{eq:mantel-balance} and \eqref{eq:mantel-L1}.

Move a set of measure $s$ from the larger part to the smaller part.  The
resulting partition is balanced, so its complete bipartite kernel can be
carried to $B_1$ by a measure-preserving rearrangement.  Its $L^1$-distance
from $K_{X,Y}$ is $2s(1-s)$.  Therefore
\begin{equation}\label{eq:mantel-final-L1}
 \inf_\varphi\lVert W^\varphi-B_1\rVert_1
 \le\eta+2r+2s-4s^2
 \le3\eta-4\eta^2+2s-4s^2.
\end{equation}
If $0\le\eta\le1/16$, then
$s\le\sqrt{\eta(1-\eta)}<1/4$.  Since $2s-4s^2$ is increasing on
$[0,1/4]$, the right side of \eqref{eq:mantel-final-L1} is at most
\[
 3\eta-4\eta^2+2\sqrt{\eta(1-\eta)}-4\eta(1-\eta)
 =2\sqrt{\eta(1-\eta)}-\eta
 \le2\sqrt\eta.
\]
If $1/16\le\eta\le1/2$, then $2s-4s^2\le1/4$, and
\[
 3\eta-4\eta^2+\frac14\le2\sqrt\eta.
\]
For the last inequality, the difference between the right and left sides
has derivative $\eta^{-1/2}-3+8\eta$.  This derivative has minimum
$3(2^{1/3}-1)>0$, attained at $\eta=2^{-8/3}$, and the difference itself
equals $5/64$ at $\eta=1/16$.  This proves the $L^1$ estimate in
\eqref{eq:mantel-cut-stability}; the cut distance estimate follows because
the cut norm is at most the $L^1$-norm.  Taking $\eta=0$ gives the equality
case of the density bound, and the converse follows by cut continuity of
$\homd(K_2,\cdot)$.
\end{proof}

\subsection{Scalar coefficient estimates}

\begin{proof}[Proof of \Cref{lem:coefficient-signs}]
Write \(s=\tau/2\) and \(s_{\mathrm c}=\tau_{\mathrm c}/2\).
The function \(\sin s/s\) is strictly decreasing on \((0,\pi)\), since
\(\sin s-s\cos s=\int_0^s u\sin u\,du>0\).
At \(s_0=\sqrt{18/5}\), its alternating series gives
\[
 \frac{\sin s_0}{s_0}
 <
 1-\frac{s_0^2}{3!}+\frac{s_0^4}{5!}
 -\frac{s_0^6}{7!}+\frac{s_0^8}{9!}
 =\frac{87361}{175000}<\frac12.
\]
Since \(\sin s_{\mathrm c}/s_{\mathrm c}=1/2\), we have
\begin{equation}\label{eq:sc-bound}
 0<s\le s_{\mathrm c}<\sqrt{\frac{18}{5}}<2.
\end{equation}

A direct simplification of \eqref{eq:hermite-coefficients} gives
\begin{align}
 \alpha_\tau&=\frac{2400}{19}
 \bigl(s^2+s\sin s-4(1-\cos s)\bigr),\label{eq:alpha-positive}\\
 \beta_\tau&=\frac8{19}
 \bigl(1200(1-\cos s)-319s\sin s-281s^2\bigr),\label{eq:beta-positive}\\
 \gamma_\tau&=4s(2\sin s-s).\label{eq:gamma-positive}
\end{align}
For the bracket in \eqref{eq:alpha-positive},
\[
 s^2+s\sin s-4(1-\cos s)
 =\sum_{k=3}^{\infty}(-1)^{k+1}
 \frac{(2k-4)s^{2k}}{(2k)!}.
\]
For \(k\ge3\), \eqref{eq:sc-bound} gives the successive ratio
\[
 \frac{(2k-2)s^{2k+2}/(2k+2)!}
      {(2k-4)s^{2k}/(2k)!}
 =\frac{k-1}{k-2}\frac{s^2}{(2k+2)(2k+1)}
 \le\frac9{70}.
\]
The alternating series begins with a positive term, so
\(\alpha_\tau>0\).

For the bracket in \eqref{eq:beta-positive}, put \(u=s^2\).  Its series is
\[
 \sum_{k=2}^{\infty}(-1)^k
 \frac{(638k-1200)s^{2k}}{(2k)!}.
\]
From \(k=5\) onward the successive absolute-term ratio is at most \(3/55\).
The \(k=6\) term is positive, so the alternating tail after the \(k=5\)
term is nonnegative.  Hence the bracket is at least
\(s^4P_\beta(u)\), where
\[
 P_\beta(u)=\frac{19}{6}-\frac{119}{120}u
 +\frac{169}{5040}u^2-\frac{199}{362880}u^3.
\]
On \(0\le u\le18/5\),
\[
 P_\beta'(u)\le-\frac{119}{120}
 +\frac{169}{2520}\frac{18}{5}<0,
 \qquad
 P_\beta(18/5)=\frac{1187}{210000}>0.
\]
Thus \(\beta_\tau>0\).  Finally, by \eqref{eq:sc-bound}, the derivative
\(2\cos s-1\) has only the zero \(\pi/3\) on \([0,s_{\mathrm c}]\).
Hence \(2\sin s-s\) first increases and then decreases to zero at
\(s_{\mathrm c}\).  Equation \eqref{eq:gamma-positive} gives
\(\gamma_\tau>0\) for \(0<\tau<\tau_{\mathrm c}\) and
\(\gamma_{\tau_{\mathrm c}}=0\).
\end{proof}

\section{Verification of the six-vertex certificate}\label{app:certificate}

This appendix records the encoding needed to reproduce the exact
verification of \cref{thm:flag-c246}.  A numerical semidefinite calculation
was used to locate the certificate, but it is not part of the proof.  The
supplementary material provides the rational certificate and a stand-alone
exact-arithmetic verifier.

For a graph on $n$ vertices, a mask records its edges in the lexicographic
order
\[
 \bigl(\{0,1\},\{0,2\},\ldots,\{0,n-1\},
       \{1,2\},\ldots,\{n-2,n-1\}\bigr).
\]
Bits are indexed from zero, and the $k$th bit corresponds to the $k$th pair
in this ordering.  An unlabeled graph is represented by the least mask in
its isomorphism class.  For the two-vertex types, flag isomorphisms fix vertices $0$ and $1$
individually; for the empty type, no vertices are fixed.
The flag masks, in matrix order, are
\begin{align*}
 \mathcal F_{\varnothing}={}&(0,1,3),\\
 \mathcal F_{\rm ne}={}&(0,2,6,8,10,12,14,24,26,30,32,34,40,42,44),\\
 \mathcal F_{\rm e}={}&(1,3,7,9,13,25,33,35,41,45).
\end{align*}
Here $\varnothing$ is the empty type, while $\sigma_{\rm ne}$ and
$\sigma_{\rm e}$ are the nonedge and edge types on two labels.  In each
product matrix, $F_i$ and $F_j$ denote flags in the corresponding displayed
order.

For a finite set $V$, let $\binom{V}{k}$ denote its family of $k$-element
subsets.  For an unlabeled six-vertex graph $H$, the empty type product
matrix is
\[
 (M_H^{\varnothing})_{ij}
 =\frac1{\binom63}\,\#\{S\in\tbinom{V(H)}3:
 H[S]\cong F_i,\ H[V(H)\setminus S]\cong F_j\}.
\]
For $\sigma\in\{\sigma_{\rm ne},\sigma_{\rm e}\}$, the corresponding
matrix is
\begin{multline*}
 (M_H^\sigma)_{ij}=\frac1{6\cdot5\binom42}
 \#\{(u,v,S):u\ne v,\ H[\{u,v\}]\cong\sigma,\\
 S\in\tbinom{V(H)\setminus\{u,v\}}2,\ 
 H[\{u,v\}\cup S]\cong F_i,\ 
 H[V(H)\setminus S]\cong F_j\}.
\end{multline*}
The isomorphisms in the second display preserve both labels.  These
normalizations are the unconditional downward averages used in
\eqref{eq:six-point-certificate}.

The certificate factors each Gram matrix as
\[
 Q_\sigma=U_\sigma R_\sigma U_\sigma^{\mathsf T}.
\]
The columns of $U_\sigma$ form a rational basis for the common nullspace of
the following row vectors, where $e_i$ is the $i$th standard coordinate
vector and indices start at zero:
\begin{align*}
 k_{\varnothing}&=(1,0,3),\\
 k_{\rm ne}^{(1)}&=e_0+e_9+2e_{13},\qquad
 k_{\rm ne}^{(2)}=e_2+e_7+2e_{14},\\
 k_{\rm e}&=e_2+e_5+2e_9.
\end{align*}
The bases are normalized so that the rows outside the pivot sets form
identity matrices; the pivot sets are
$\{2\}$, $\{13,14\}$, and $\{9\}$ for the three types, respectively.
The reduced Gram matrices $R_\sigma$, of orders $2$, $13$, and $9$, together
with the $38$ residuals $\eta_H$, are included in the supplementary
certificate.  Exact $LDL^{\mathsf T}$ elimination gives the following
smallest pivots.
\[
\begin{array}{@{}lccc@{}}
\toprule
\text{type}&\varnothing&\sigma_{\rm ne}&\sigma_{\rm e}\\
\midrule
\text{order of }R_\sigma&2&13&9\\
\text{smallest pivot}
&349223/22485600
&6469543/87833875
&1327666493/21687735285\\
\bottomrule
\end{array}
\]
The complete pivot lists are checked by the verifier; every pivot is
positive.  Seven residuals vanish, and the remaining $31$ are positive,
with minimum $107/5400$.

A stand-alone verifier supplied with the supplementary material enumerates
the $5\,789$ labeled triangle-free graphs on six vertices, reduces them to
$38$ isomorphism classes, reconstructs the three flag-product matrices,
checks all identities in \eqref{eq:six-point-certificate}, and verifies the
positive definiteness of the reduced matrices $R_\sigma$ by exact
$LDL^{\mathsf T}$ elimination.  All computations are performed using exact
rational arithmetic.  Thus the verification is independent of the numerical
semidefinite calculation used to locate the certificate.

\section*{Data and code availability}
The exact rational certificate and the stand-alone verification program
supporting \cref{thm:flag-c246} are supplied as supplementary material.
The program uses only the Python standard library.  No external datasets
were used.

\section*{Declarations}
The author declares no competing interests.

\end{document}